\documentclass[11pt]{article}
\usepackage[a4paper,margin=2.5cm]{geometry}
\usepackage[comma, round]{natbib}
\usepackage{float}
\usepackage{amssymb}
\usepackage{amsthm}
\usepackage{amsmath}
\usepackage{titletoc}
\usepackage{mathrsfs}
\usepackage[colorlinks, allcolors=blue]{hyperref}
\usepackage{amsfonts}
\usepackage{graphicx}
\usepackage[linesnumbered, ruled, vlined]{algorithm2e}
\usepackage{esint}
\usepackage{bbm}
\usepackage{bm}
\usepackage{mathtools}
\usepackage[shortlabels]{enumitem}
\usepackage{authblk}

\numberwithin{equation}{section}
\numberwithin{figure}{section}

\theoremstyle{plain}
\newtheorem{thm}{Theorem}[section]

\newtheorem{lem}[thm]{Lemma}
\newtheorem{prop}[thm]{Proposition}

\theoremstyle{definition}
\newtheorem{defn}{Definition}[section]
\newtheorem{asmp}{Assumption}[section]
\newtheorem{rmk}[defn]{Remark}
\newtheorem{exam}{Example}[section]

\newcommand{\cB}{\mathcal{B}}

\newcommand{\cD}{\mathcal{D}}

\newcommand{\cF}{\mathcal{F}}

\newcommand{\cH}{\mathcal{H}}

\newcommand{\cK}{\mathcal{K}}

\newcommand{\cP}{\mathcal{P}}

\newcommand{\cT}{\mathcal{T}}

\newcommand{\cX}{\mathcal{X}}
\newcommand{\cY}{\mathcal{Y}}

\newcommand{\bF}{\mathbf{F}}

\newcommand{\bbH}{\mathbb{H}}

\newcommand{\bbR}{\mathbb{R}}

\newcommand{\bbX}{\mathbb{X}} 
\newcommand{\bbY}{\mathbb{Y}} 
\newcommand{\bbZ}{\mathbb{Z}}

\newcommand{\scrM}{\mathscr{M}}

\newcommand{\scrP}{\mathscr{P}}

\newcommand{\scrT}{\mathscr{T}}

\newcommand{\1}{\mathbbm{1}}
\newcommand{\AW}{\mathcal{AW}}

\newcommand{\md}{\mathop{}\mathopen\mathrm{d}}

\newcommand{\E}{E}
\renewcommand{\P}{\mathbb{P}}

\newcommand{\Id}{\operatorname{Id}}

\newcommand{\proj}{\operatorname{pj}}

\newcommand{\ad}{\mathscr{U}}

\newcommand{\dom}{\operatorname{dom}}
\newcommand{\Law}{\operatorname{Law}}

\newcommand{\AP}{\mathsf{AP}}
\newcommand{\NP}{\mathsf{NP}}
\newcommand{\bav}{\mathbf{ES}_{\alpha}}
\newcommand{\av}{\mathrm{ES}_{\alpha}}
\renewcommand{\c}{\mathfrak{c}}
\newcommand{\bc}{\mathfrak{bc}}

\DeclareMathAlphabet{\mathbbmsl}{U}{bbm}{m}{sl}

\newcommand{\Q}{Q}
\newcommand{\tr}{\operatorname{tr}}
\def\lb{\mathopen{}\mathclose\bgroup\left}
\def\rb{\aftergroup\egroup\right}

\newcommand{\os}{\mathrm{OS}}
\newcommand{\bos}{\mathbf{OS}}
\newcommand{\bbos}{\mathbb{OS}}

\title{Duality of causal distributionally robust optimization}
\author{Yifan Jiang \thanks{Email: {\tt yifan.jiang@imperial.ac.uk}.}}
\affil{Department of Mathematics, Imperial College London}
\date{}

\begin{document}
\maketitle

\begin{abstract}
	We study distributionally robust optimization (DRO) in a dynamic context, where model uncertainty is captured by penalizing potential models based on their adapted Wasserstein distance to a reference model.
	We consider both discrete- and continuous-time settings and derive dynamic duality formulas that reformulate the worst-case expectation as a tractable minimax problem.
	The inner maximization admits a recursive representation in discrete time, while in continuous time, it is characterized by a path-dependent Hamilton--Jacobi--Bellman equation.
	We further extend these duality results from the worst-case expectation to the worst-case expected shortfall, a non-linear expectation.
	Finally, we apply this framework to optimal stopping problems in discrete time.
	We recast the original problem as a classical Wasserstein DRO on a nested space by introducing a novel relaxation that considers stopping times with respect to general filtrations.

	\medskip

	\medskip
\end{abstract}

\section{Introduction}
As Box famously said in \citet{box76Science}, \emph{all models are wrong, but some are useful}.
In real-world applications, practitioners often need to trade off between a model's fidelity (its capability to capture system features) and a model's tractability (its capability to provide interpretable solutions).
A postulated model may capture some important aspects of reality, while inevitably ignoring some other aspects.
This is especially true in mathematical finance, where the model is often derived from theoretical considerations, possibly combined with some calibration to market data.
This inherent Knightian uncertainty \citep{knight21Risk} regarding the model is of fundamental importance and a subject of intense study in mathematics and economics alike.
Mathematical frameworks such as risk measures \citep{follmer08Stochastic} and sublinear expectations \citep{peng19Nonlinear} have been developed to take into account such uncertainty.
More recently, Wasserstein distributionally robust optimization (W-DRO)  has emerged as a powerful tool to counter model uncertainty, for example in \citet{mohajerinesfahani18Datadriven,blanchet19Quantifying,gao23Distributionally}.
It is formulated as a minimax problem  that optimizes a worst-case objective, evaluated over a collection of models.
This collection is often referred to as the ambiguity set and is taken as a Wasserstein ball centered at a reference model.

In this paper, we extend Wasserstein distributionally robust optimization (W-DRO) to a dynamic setting, where model uncertainty is quantified by the causal or the adapted Wasserstein distance.
This distance not only captures the spatial differences between two models but also the information flow they generate.
The choice of ambiguity set here is natural in a dynamic setting and, to some extent, becomes essential when considering optimal stopping problems.
Crucially, it allows for models with potentially different support from the reference model, similar to the classical W-DRO, while excluding models that can only be obtained from an \emph{anticipative} perturbation of the reference model.

Despite its appealing theoretical  framework, causal distributionally robust optimization  often suffers from the computational challenges  inherited from  causal optimal transport.
Our contribution here is to address this gap by providing a tractable dynamic duality formula for the causal DRO problem in both discrete- and continuous-time settings.
We focus on the distributional model risk:
\begin{equation*}
	\sup_{\nu\in B_{\delta}(\mu)}\E_{\nu}[f(X)],
\end{equation*}
where \(B_{\delta}(\mu)\) is an ambiguity set given by a causal Wasserstein ball or an adapted Wasserstein ball.
We extend this hard ambiguity set constraint to a more general penalized form
\begin{equation}
	\label{eqn-4-dro}
	V_{\star}=\sup_{\nu\in\scrP(\cX)}\{\E_{\nu}[f(X)]-L(\cT_{\star }(\mu,\nu))\},
\end{equation}
where \(L\) is a penalty function, and \(\cT_{\star }\) represents the optimal causal/bi-causal transport cost given by
\begin{equation*}
	\cT_{\star}(\mu,\nu)= \inf_{\pi\in \Pi_{\star}(\mu,\nu)} \E_{\pi}[c(X,Y)],
\end{equation*}
with \(\star\in \{\c,\bc\}\) respectively.
By choosing  an indicator penalization \(L\) and an appropriate cost \(c\), one  recovers the corresponding ambiguity set \(B_{\delta}(\mu)\).

While \(V_{\c}\) and \(V_{\bc}\) are not equal a priori, we will show that under mild regularity conditions, these two penalizations are equivalent.
This allows us to omit the subscript when there is no ambiguity.
Our main results, detailed in Sections~\ref{sec-4-disc} and~\ref{sec-4-cont}, show that
\begin{equation*}
	V= \inf_{\lambda\geq 0}\{L^{*}(\lambda)+U(\lambda)\},
\end{equation*}
where \(L^{*}\) is the convex conjugate of \(L\), and \(U\) is a convex function given  by a dynamic programming principle.
In particular, \(U\) can be computed recursively in discrete time, or solved by a path-dependent Hamilton--Jacobi--Bellman (HJB) equation in continuous time.
Moreover, in Sections~\ref{sec-4-es} and~\ref{sec-4-os} we replace the linear expectation \(\nu\mapsto \E_{\nu}[f(X)]\) in \eqref{eqn-4-dro} with a nonlinear functional \(\nu\mapsto F(\nu)\) and derive duality formulas in discrete time.
In Section~\ref{sec-4-es}, we take \(F\) to be the expected shortfall, a concave functional of the model;
in Section~\ref{sec-4-os}, \(F\) is taken as the value of an optimal stopping problem, which is in general not concave.

To the best of our knowledge, this is the first result that tackles the duality of causal DRO in continuous time and considers optimal stopping problems.
In continuous time, we reformulate the dual problem as a non-Markovian stochastic control problem and  utilize the tools from functional It\^o calculus \citep{cont10Change,bally16Stochastic} to establish a path-dependent HJB equation.
For the causal distributionally robust optimal stopping problems, we leverage a novel concave relaxation which lifts the optimal stopping problem to the space of adapted stochastic processes.
The original problem is then recast as a Wasserstein DRO problem on the `nested space', see Definition \ref{defn-4-nested} below.

We conclude the introduction by comparing the Wasserstein DRO and the causal DRO. Generally, no universal rule dictates the preference of one over the other.
How to model the ambiguity set depends on the decision maker's uncertainty aversion.
Addressing this modeling choice is beyond our current scope.
Instead, we focus on providing a computationally feasible method for decision makers specifically interested in countering model uncertainty from non-anticipative perturbations.

We note, however, that these frameworks coincide in a special case.
If the payoff \(f\) and the transport cost \(c\) are time-separable, Theorem \ref{thm-4-disc} implies that the Wasserstein DRO is equivalent to the causal formulation.
Conversely, in general settings, enforcing causal constraints typically yields a smaller ambiguity set.
This reduction can mitigate the over-conservatism often observed in standard Wasserstein DRO in a dynamic context.
We refer readers to Example \ref{examp-es} for a detailed numerical comparison.
This distinction is particularly critical in optimal stopping problems, where the value function is known to be discontinuous with respect to the classical Wasserstein topology, see \citet{backhoff-veraguas20All}.
Thus, the causal DRO framework is necessary in such contexts to provide a reasonable distributional model risk.

\subsection{Related literature}
Distributionally robust optimization (DRO) provides a framework for decision-making when the underlying stochastic model is uncertain.
Unlike classical stochastic optimization, which assumes a single reference distribution, DRO considers a collection of plausible distributions, often defined via statistical distances such as the Wasserstein distance or \(\phi\)-divergences and optimizes the worst-case expectation over this ambiguity set.
This review is focused on the transport-type DRO; for a broader survey of the field, we refer interested readers to \citet{rahimian22Frameworks,kuhn25Distributionally}.

Wasserstein DRO has found broad applications in operations research, mathematical finance \citep{bartl20Computational,blanchet22Distributionally}, and machine learning \citep{bai23Wasserstein, bai25Wasserstein,blanchet19Quantifying}, where robustness to model misspecification is crucial.
A key development in Wasserstein DRO is its duality theory, which reveals a close connection to regularized optimization.
This theory has been progressively generalized, starting from the data-driven case where the reference measure is an empirical measure in \citet{mohajerinesfahani18Datadriven},
extending to Borel measures on Euclidean spaces in \citet{gao23Distributionally},
to general Polish spaces in \citet{blanchet19Quantifying},
and to spaces with the interchangeability property, such as Suslin spaces in \citet{zhang24Short}.
More recently, variants of the classical Wasserstein DRO have also been introduced.
These include the robust optimized certainty equivalents in a  penalized form \citep{bartl20Computational}, a weak optimal transport-type DRO \citep{kupper23Risk}, and problems with marginal uncertainty in both source and target distributions \citep{fan23Quantifying}.
A parallel approach studies Wasserstein DRO via sensitivity analysis with respect to model uncertainty.
It was first established in \citet{bartl21Sensitivity} and extended to dynamic settings in \citet{bartl23Sensitivity,sauldubois24First,jiang24Sensitivity}.

Few duality results are available for the causal DRO in a dynamic context, with all existing work focusing solely on a discrete-time setting.
Handling the causality constraint poses a significant theoretical challenge.
In a recent work \citep{han25Distributionally}, the author reformulates the causality constraint as an infinite-dimensional linear constraint, recasting the original problem as an optimization over a space of test functions.
Furthermore, by employing adapted empirical measures, the author derives  sample complexity for the discrete-time causal DRO.
This provides a theoretical guarantee for neural network approximations.
In contrast, our proposed dynamic duality formula offers a tractable solution that decomposes the global problem into a sequence of single-step optimization tasks.
It not only accommodates a more general penalized form \eqref{eqn-4-dro}, but critically, it leverages the temporal structure of causal couplings.
We remark that our discrete-time duality Theorem~\ref{thm-4-disc} first appeared in \citet{jiang23Adapted}.
This result has also been independently obtained in \citet{gao22Datadriven,yang22Decisionmaking} under an indicator penalization and stronger regularity constraints.

Our continuous-time results rely on tools from functional It\^o calculus, which were first proposed by \citet{dupire09Functional}, and systematically studied in \citet{cont10Change,cont13Functional}.
It was then applied to study non-Markovian stochastic control problems and their associated path-dependent HJB equations.
A verification theorem  for the classical solution was established  in \citet[Chapter 8.3]{bally16Stochastic}.
The viscosity solution theory, however, has proven more intricate, leading to several proposed notions, for example in \citet{tang15Pathdependent,ekren14viscosity,ekren16Viscosity,ekren16Viscositya}, etc.

\subsection{Outline}
The rest of the paper is organized as follows.
In Section~\ref{sec-pre}, we introduce the basic notation and tools.
In Section~\ref{sec-4-dual}, as an intermediate step, we derive a general duality formula for causal DRO problems where the penalty is given as a function of a causal optimal transport problem.
The dual problem involves a causal transport problem with a fixed reference distribution only.
A generalized Fenchel--Moreau duality theorem is proved in Lemma~\ref{lem-4-dual}.
In Section~\ref{sec-4-disc}, we focus on the discrete-time setting and derive a dynamic duality formula in Theorem~\ref{thm-4-disc}.
Under a mild continuity condition, the equivalence between the causal penalization and the bi-causal penalization is established in Theorem~\ref{thm-4-reg}.
In Section~\ref{sec-4-cont}, we consider a continuous-time setting with a penalty given by the Cameron--Martin adapted Wasserstein distance.
We reformulate the dual problem as a stochastic control problem and identify the worst-case distribution via a path-dependent HJB equation in Theorem~\ref{thm-4-cont}.
In Sections~\ref{sec-4-es} and~\ref{sec-4-os}, we extend Theorem~\ref{thm-4-disc} in two directions.
We replace the linear expectation with the expected shortfall in Section~\ref{sec-4-es}, and study a numerical example of a two-step exotic option.
In Section~\ref{sec-4-os}, we consider  optimal stopping problems in the causal DRO framework.
A duality formula is derived in Theorem~\ref{thm-4-os} by lifting the original problem to the space of adapted processes.

\section{Preliminaries}
\label{sec-pre}
Let \(N\in \bbZ_{+}\) be the number of steps.
For any \(1\leq n\leq m\leq N\) and an \(N\)-tuple \((\theta_{1},\dots,\theta_{N})\), we denote the truncation \((\theta_{n},\dots,\theta_{m})\) by \(\theta_{n:m}\).
Let \(\mathbb{R}^{*}=\mathbb{R}\cup\{+\infty\}\) be the extended real line.
For \(A,B\subseteq \mathbb{R}\), we define the sum set \(A+B:=\{a+b : a\in A ,b\in B\}\).

Let \(C_{0}([0,T]; \mathbb{R}^{d})\) be the continuous path space starting at origin.
By  \(\|\cdot\|_{\infty}\) we denote the uniform norm, and by \(\|\cdot\|_{H^{1}_{0}}\) we denote the Cameron--Martin norm given by
\(
\|x\|_{H^{1}_{0}}:=\bigl(\int_{0}^{T}\|\dot{x}(t)\|^{2}\md t\bigr)^{1/2}.
\)

For a generic Polish space \(\cX\), we equip it with its Borel \(\sigma\)-algebra \(\cB(\cX)\).
Let \(\scrP(\cX)\) be the space of Borel probability measures on \(\cX\) equipped with its weak topology.
We denote the completion of \(\sigma\)-algebra \(\mathcal{F}\subseteq \mathcal{B}(\mathcal{X})\)  under \(\mu\in \mathscr{P}(\mathcal{X})\) by \(\prescript{\mu}{}{}\mathcal{F}\).
For any \(\mu\in\scrP(\cX)\) and an integrable function \(\varphi:\cX\to \bbR\), we use the shorthand notation
\begin{equation*}
	\mu(\varphi):=\E_{\mu}[\varphi]=\int \varphi\md \mu.
\end{equation*}
We follow the convention that \(\infty-\infty=-\infty\).
In particular, for any measurable \(f\), we define \(\E_{\mu}[f]=\E_{\mu}[\max\{f,0\}]+ \E_{\mu}[\min\{f,0\}]\).
Given \(\mu,\nu\in\scrP(\cX)\),  the set of couplings between \(\mu\) and \(\nu\) is defined as
\begin{equation*}
	\Pi(\mu,\nu):=\{\pi\in\scrP(\cX\times\cX):\pi(\cdot\times\cX)=\mu(\cdot)\text{ and } \pi(\cX\times\cdot)=\nu(\cdot)\}.
\end{equation*}
For any \(\pi\in \scrP(\cX\times \cX)\), we write
\begin{equation*}
	\pi(\md x,\md y)=\pi(\md x) \pi_{x}(\md y),
\end{equation*}
where  \(\pi_{x}\) is the Borel regular disintegration kernel.

\subsection{Causal optimal transport}
Following \citet{acciaio20Causal,backhoff-veraguas17Causal,bartl24Wasserstein}, we give a brief introduction to causal optimal transport.
We are interested in the probability measures on path spaces which are interpreted as laws of stochastic processes.
In discrete time,  we take the time index set \(I\) as \(\{0,1,\dots,N\}\) and the canonical path space \(\cX=\cX_{0}\times\cX_{1}\times\cdots\times\cX_{N}\) the product of Polish spaces;
in continuous time, we consider \(I=[0,T]\) and  \(\cX=C_{0}([0,T];\bbR^{n})\) the continuous path space.
We equip the path space \(\cX\) with the uniform norm in both cases.
Let \(\Id_{\cX}:\cX\to\cX\) be the identity map on \(\cX\).
It is interpreted as the  canonical process and induces a natural filtration \(\bF=(\cF_{t})_{t\in I}\) given by \(\cF_{t}=\sigma(\Id_{\cX}(s):0\leq s\leq t)\), which denotes the information available at time \(t\).

To incorporate the information structure of the underlying space, the classical optimal transport is modified using the key concept of a \emph{causal} coupling.
Heuristically, this is a transport plan that is constrained to move mass without using any information from the future.

\begin{defn}[Causal coupling]
	\label{defn-2-causal-coupling}
	Let \(\mu,\nu\in\scrP(\cX)\).
	We say a coupling \(\pi\in\Pi(\mu,\nu)\) is \emph{causal} if \[x_{1}\mapsto \pi_{x_{1}}(V_{t}) \text{ is } \prescript{\mu}{}{}\cF_{t}\text{-measurable}\] for any \(V_{t}\in\cF_{t}\) and \(t\in I\), where \(\prescript{\mu}{}{}\cF_{t}\) is the completion of \(\cF_{t}\) under \(\mu\) and \(\pi_{x_{1}}\) is the disintegration kernel given by \(\pi(\md x_{1},\md x_{2})=\mu(\md x_{1}) \pi_{x_{1}}(\md x_{2})\).
	A causal coupling \(\pi\) is \emph{bi-causal} if further \([(x_{1},x_{2})\mapsto (x_{2},x_{1})]_{\#}\pi\) is causal.
	We write the set of causal (bi-causal) couplings between \(\mu\) and \(\nu\) as \(\Pi_{\c}(\mu,\nu)\) \((\Pi_{\bc}(\mu,\nu))\).
\end{defn}
We remark that an equivalent definition in discrete time can be found in Proposition \ref{prop-4-dual}.

Given a cost function \(c:\mathcal{X}\times \mathcal{X} \to \mathbb{R}\), the corresponding (bi-)causal optimal transport is given by
\begin{equation*}
	\mathcal{T}_{\star}(\mu,\nu)=\inf_{\pi\in \Pi_{\star}(\mu,\nu)}\E_{\pi}[c(X,Y)],
\end{equation*}
where \(\star=\c/\bc\).
By \(\Pi(\mu,*)\) we denote the set of couplings with a fixed first marginal \(\mu\).
Accordingly,  \(\Pi_{\c}(\mu,*)\) and \(\Pi_{\bc}(\mu,*)\) represent the respective subsets of causal and bi-causal couplings.

We then introduce the causal coupling between two adapted stochastic processes equipped with general filtrations.

\begin{defn}
	We say an adapted stochastic process \(\bbX\) is  given by a 5-tuple \((\Omega^{\bbX}, \cF^{\bbX}, \P^{\bbX}, \bF^{\bbX}, X)\) where \((\Omega^{\bbX}, \cF^{\bbX}, \bF^{\bbX},\P^{\bbX})\) is a filtered Polish probability space, and \(X:\Omega^{\bbX} \to \cX\) is a stochastic process adapted to the filtration \(\bF^{\bbX}=(\cF_{t}^{\bbX})_{t\in I}\).
\end{defn}
For an adapted process \(\bbX= (\Omega^{\bbX}, \cF^{\bbX}, \P^{\bbX}, \bF^{\bbX}, X)\), if \(\bF^{\bbX}\) coincides with \(\bF^{X}\), the natural filtration generated by \(X\), then we say \(\bbX\) is a \emph{naturally filtered} process.
With a slight abuse of notation, we may also use \(X\) to refer to the 5-tuple of a naturally filtered process since itself determines the natural filtration.
We denote the space of  adapted processes  and naturally filtered processes by \(\AP\) and \(\NP\) respectively.

The notion of causality can be naturally carried over to adapted stochastic processes as follows.

\begin{defn}[Causal coupling]
	\label{defn-2-causal-coupling2}
	Let \(\bbX_{i}=(\Omega^{\bbX_{i}},\cF^{\bbX_{i}},\bF^{\bbX_{i}},\P^{\bbX_{i}},X_{i})\) be two adapted stochastic processes for \(i=1,2\).
	We say a coupling \(\pi\in \Pi(\P^{\bbX_{1}},\P^{\bbX_{2}})\) is \emph{causal} if
	\begin{equation*}
		\omega_{1}\mapsto \pi_{\omega_{1}}(V_{t}) \text{ is } \prescript{\P^{\bbX_{1}}}{}{}\cF_{t}^{\bbX_{1}}\text{-measurable}    \end{equation*}
	for any \(V_{t}\in \cF_{t}^{\bbX_{2}}\) and \(t\in I\), where \(\pi_{\omega_{1}}\) is the disintegration kernel of \(\pi\) with respect to \(\Omega^{\bbX_{1}}\).
	A causal coupling \(\pi\) is \emph{bi-causal} if further \([(\omega_{1},\omega_{2})\mapsto (\omega_{2},\omega_{1})]_{\#}\pi\) is causal.
	With a slight abuse of notation, we write the set of causal (bi-causal) couplings between \(\bbX_{1}\) and \(\bbX_{2}\) as \(\Pi_{\c}(\bbX_{1},\bbX_{2})\) (\(\Pi_{\bc}(\bbX_{1},\bbX_{2})\)).
\end{defn}

\subsection{Analytic sets and universal measurability}
The analytic sets (also known as Suslin sets) are widely applied to reconcile the measurability issue in the dynamic programming principle.
We give a minimal introduction here to serve our purpose.

\begin{defn}
	Let \(\cX\) be a Polish space.
	We say  a subset  \(S\subseteq\cX\) is  analytic  if \(S\) is a continuous image of a Polish space, and a function \(\varphi:\cX\to\bbR\) is upper semi-analytic if the level set \(\{\varphi\geq r\}\) is analytic for any \(r\in\bbR\).
	The universal \(\sigma\)-algebra of \(\cX\) is defined as
	\(
	\bigcap_{\mu\in\scrP(\cX)}  \prescript{\mu}{}\cB(\cX).
	\)
\end{defn}
To ease the notation, we will not distinguish any Borel measure \(\mu\) and its completion.
For any universally measurable function \(\varphi\), \(\mu(\varphi)\) is understood as the integration with respect to the completion of \(\mu\).
We recall the following propositions from \citet[Corollary 7.42.1, Proposition 7.39, and Proposition 7.50]{bertsekas96Stochastic}.
\begin{prop}
	\label{prop-ana}
	Let \(S\subseteq \cX\) be an analytic set.
	Then \(S\) is universally measurable, and therefore any upper semi-analytic function is universally measurable.
\end{prop}

\begin{prop}
	\label{prop-proj}
	Let \(\cX\) and \(\cY\) be Polish spaces and \(D\subseteq \cX\times \cY\) an analytic set.
	Then the projection \(\proj_{\cX}(D):=\{x\in\cX:(x,y)\in D \text{ for some } y\}\) is an analytic subset of \(\cX\).
\end{prop}
\begin{prop}[Analytic selection theorem]
	\label{prop-sel}
	Let \(\cX\) and \(\cY\) be Polish spaces and \(D\subseteq \cX\times \cY\) an analytic set, and \(\varphi:D\to \bbR\) an upper semi-analytic function.
	Define \(\tilde{\varphi}:\proj_{\cX}(D)\to \bbR^{*}\) by
	\begin{equation*}
		\tilde{\varphi}(x)=\sup_{y\in D_{x}}\varphi(x,y),
	\end{equation*}
	where \(D_{x}=\{y\in \cY:(x,y)\in D\}\).
	Then for any  \(\varepsilon>0\), there exists an analytically measurable function \(s:\proj_{\cX}(D)\to \cX\) such that \((x,s(x))\in D\) for any \(x\in \proj_{\cX}(D)\), and
	\begin{equation*}
		\varphi(x,s(x))\geq\begin{cases}
			\tilde{\varphi}(x)-\varepsilon & \text{if}\quad \tilde{\varphi}<\infty,  \\
			1/\varepsilon \quad            & \text{if} \quad \tilde{\varphi}=\infty.
		\end{cases}
	\end{equation*}
\end{prop}

The following result is adapted from \citet{zhang24Short}.
\begin{lem}
	\label{lem-static}
	Let \((\cX,\cB)\) be a Polish space equipped with its Borel \(\sigma\)-algebra, \(\varphi:\cX\times\cX\to \bbR\) upper semi-analytic.
	Define  \(\tilde{\varphi}:\cX \to \bbR^{*}\) as
	\begin{equation*}
		\tilde{\varphi}(x)=\sup_{y\in\cX}\varphi(x,y).
	\end{equation*}
	Then \(\tilde{\varphi}\) is universally measurable; moreover, it holds that
	\begin{equation*}
		\E_{\mu}[\sup_{y\in \cX}\varphi(X,y)]=\sup_{\pi\in \Pi(\mu,*)}\E_{\pi}[\varphi(X,Y)].
	\end{equation*}
\end{lem}
\begin{proof}
	Notice that for any \(r\in \bbR\) we have
	\begin{equation*}
		\bigl\{x:\sup_{y\in\cX}\varphi(x,y)>r\bigr\}=\proj_{\cX}(\{(x,y):\varphi(x,y)>r\}).
	\end{equation*}
	Since \(\varphi\) is upper semi-analytic, we have that \(\tilde{\varphi}\) is also upper semi-analytic by Proposition~\ref{prop-proj}.
	Moreover, by Proposition~\ref{prop-ana}, we know \(\tilde{\varphi}\) is universally measurable.
	It follows from the definition that
	\begin{equation*}
		\E_{\mu}[\sup_{y\in\cX}\varphi(X,y)]\geq \sup_{\pi\in\Pi(\mu,*)}\E_{\pi}[\varphi(X,Y)].
	\end{equation*}

	Now, by Proposition~\ref{prop-sel} there exists an analytically measurable function \(s_{n}:\cX\to \cX\) such that
	\begin{equation*}
		\varphi(x,s_{n}(x))\geq \begin{cases}
			\tilde{\varphi}(x)-\frac{1}{n}\quad & \text{if} \quad \tilde{\varphi}(x)<\infty, \\
			n\quad                              & \text{if} \quad \tilde{\varphi}(x)=\infty.
		\end{cases}
	\end{equation*}
	Moreover, there exists a Borel measurable function \(t_{n}\) such that \(t_{n}=s_{n}\) holds \(\mu\)-a.s.
	Then we take
	\(
	\pi_{n}=(\Id,t_{n})_{\#} \mu\in \Pi(\mu,*).
	\)
	We derive
	\begin{equation*}
		\quad \E_{\pi_{n}}[\varphi(X,Y)]=\E_{\mu}[\varphi(X,t_{n}(X))] \geq \E_{\mu}[\tilde{\varphi}\1_{\{\tilde{\varphi}<\infty\}}]+n\E_{\mu}[\1_{\{\tilde{\varphi}=\infty\}}]-\frac{1}{n}.
	\end{equation*}
	As \(n\to\infty\), we conclude
	\begin{equation*}
		\sup_{\pi\in\Pi(\mu,*)}\E_{\pi}[\varphi(X,Y)]\geq \E_{\mu}[\sup_{y\in\cX}\varphi(X,y)].
	\end{equation*}
\end{proof}

\subsection{Convex analysis}
We recall several basic concepts and results in convex analysis.
\begin{defn}
	Let \(\varphi:\bbR \to \bbR^{*}\) be a convex function.
	We define the \emph{domain} of \(\varphi\) as
	\begin{equation*}
		\dom(\varphi)=\{x\in\bbR:\varphi(x)<\infty\}.
	\end{equation*}
	We say \(\varphi\) is \emph{proper} if \(\dom(\varphi)\neq \emptyset\), and \(\varphi\) is \emph{closed} if \(\dom(\varphi)\) is closed.
	A function \(\psi\) is proper closed concave if and only if \(-\psi\) is proper closed convex.
\end{defn}

\begin{defn}
	Let \(\varphi:\bbR\to\bbR^{*}\) be a convex function.
	The convex conjugate of \(\varphi:\bbR\to \bbR^{*}\) is defined as
	\(
	\varphi^{*}(y)=\sup_{x\in\bbR}\{xy-\varphi(x)\}.
	\)
	Similarly, we define the concave conjugate of \(\psi\) as
	\(
	\psi_{*}(y)=\inf_{x\in\bbR}\{xy-\psi(x)\}.
	\)
\end{defn}

The following result is the celebrated convex duality theorem from \citet[Corollary 12.2.1]{rockafellar97Convex}.
\begin{thm}[Fenchel--Moreau Theorem]
	\label{thm-4-fenchel}
	Let \(\varphi:\bbR\to \bbR^{*}\) be a closed proper convex (concave) function.
	Then, we have  \(\varphi^{**}=\varphi\) (\(\varphi_{**}=\varphi\)).
\end{thm}

We introduce the subdifferential of a convex function as an extension of the classical derivative to the convex functions which are not necessarily differentiable.
\begin{defn}[Subdifferential]
	\label{defn-4-subdiff}
	Let \(\varphi:\bbR\to \bbR^{*}\) be a convex function.
	For any \(x\in\bbR\), we define the \emph{subdifferential} of \(\varphi\) at \(x\) as
	\begin{equation*}
		\partial \varphi(x)=\{y\in\bbR:\varphi(x')\geq \varphi(x)+y(x'-x)\quad \forall x'\in\bbR\}.
	\end{equation*}
\end{defn}

The following proposition characterizes the subdifferential of a proper, closed, convex function, see \citet[Theorem 23.5]{rockafellar97Convex}.
\begin{prop}
	\label{prop-4-subdiff}
	Let \(\varphi:\bbR\to \bbR^{*}\) be a convex function.
	If \(\varphi\) is proper and closed, then \(y\in \partial \varphi(x)\) if and only if \(x\in\partial \varphi^{*}(y)\).
	Moreover, we have the equality \(\varphi(x)+\varphi^{*}(y)=xy\).
\end{prop}

\subsection{Horizontal and vertical derivatives}
Following \citet{dupire09Functional,cont10Change,bally16Stochastic}, we  introduce the horizontal and vertical derivatives of a non-anticipative functional.
We start with the definition of a non-anticipative functional on the c\`adl\`ag path space \(D([0,T];\bbR^{n})\).
\begin{defn}
	\label{def-nonanti}
	We say a functional \(F:[0,T]\times D([0,T];\bbR^{n})\to\bbR\) is non-anticipative if
	\(F(t,x)=F(t,x(\cdot\wedge t))\) for any \(t\in[0,T]\) and \(x\in D([0,T];\bbR^{n})\).
\end{defn}

\begin{defn}
	\label{defn-hori}
	A non-anticipative functional \(F\) is said to be \emph{horizontally} differentiable at \((t,x)\in [0,T]\times D([0,T];\bbR^{n})\) if the limit
	\begin{equation*}
		\cD F(t,x)=\lim_{h\to 0^{+}}\frac{F(t+h,x(\cdot\wedge t))-F(t,x(\cdot\wedge t))}{h} \text{ exists.}
	\end{equation*}
	We call \(\cD F(t,x)\) the \emph{horizontal} derivative of \(F\) at \((t,x)\).
\end{defn}

\begin{defn}
	\label{defn-vert}
	A non-anticipative functional \(F\) is said to be \emph{vertically} differentiable at \((t,x)\in [0,T]\times D([0,T];\bbR^{n})\) if the map
	\begin{equation*}
		\bbR^{n}\ni e\mapsto F(t,x(\cdot\wedge t)+e \1_{[t,T]}) \in \bbR
	\end{equation*}
	is differentiable at \(0\).
	We call its gradient at \(0\) the \emph{vertical} derivative of \(F\) at \((t,x)\) and denote it by \(\nabla_{x}F(t,x)\).
\end{defn}

Before proceeding to non-anticipative functionals on the  continuous path space, we introduce some regularity conditions.
We equip the product space \([0,T]\times D([0,T];\bbR^{n})\) with the metric \[d_{\infty}((t,x),(t',x')):= |t-t'|+\|x(\cdot \wedge t)-x'(\cdot\wedge t')\|_{\infty}.\]

\begin{defn}[Left continuity]
	We say a non-anticipative functional \(F\) is left continuous if for any \((t,x)\in [0,T]\times D([0,T];\bbR^{n})\) and any sequence \((t_{n},x_{n})\stackrel{d_{\infty}}{\to}(t,x)\) in \([0,T]\times D([0,T];\bbR^{n})\) with \(t_{n}\) increasing, we have
	\begin{equation*}
		F(t_{n},x_{n})\to F(t,x).
	\end{equation*}
	By \(C_{l}\) we denote the space of left continuous non-anticipative functionals which are also continuous in \(x\) at any fixed time \(t\in [0,T]\).
\end{defn}

\begin{defn}[Boundedness preserving]
	We say a non-anticipative functional \(F\) is boundedness preserving if for any  \(M>0\) and \(t_{0}<T\) there exists \(C_{M,t_{0}}\) such that
	\begin{equation*}
		|F(t,x)|\leq C_{M,t_{0}} \quad \text{ if } \quad \|x(\cdot\wedge t_{0})\|_{\infty}\leq M.
	\end{equation*}
	By \(C_{b}\) we denote the space of  boundedness preserving non-anticipative functionals.
\end{defn}

The following class of regular non-anticipative functionals plays a special role in functional It\^o calculus.
\begin{defn}
	We say a non-anticipative functional \(F\) is in the class \(C^{1,2}_{b}\) if the following conditions hold:
	\begin{enumerate}[label=(\roman*)]
		\item \(F\) admits a horizontal derivative and the first- and second-order vertical derivatives for any \((t,x)\in [0,T]\times D([0,T];\bbR^{n})\).
		\item \(F,\cD F, \nabla_{x}F, \nabla_{x}^{2}F\in C_{l}\).
		\item \(\cD F, \nabla_{x}F, \nabla^{2}_{x}F \in C_{b}\).
	\end{enumerate}
\end{defn}

We recall \citet[Theorem 5.27]{bally16Stochastic} and \citet[Theorem 5.28]{bally16Stochastic} as follows.
\begin{thm}	Let \(F_{1},F_{2}\) be two non-anticipative functionals in \(C^{1,2}_{b}\).
	If \(F_{1}\) and \(F_{2}\) are equal on all continuous paths,
	then their first and second horizontal derivatives coincide on all continuous paths.
\end{thm}

We say a non-anticipative functional \(F:[0,T]\times C([0,T];\bbR^{n})\) is of class \(C^{1,2}_{b}\) if there exists a non-anticipative functional \(G:[0,T]\times D([0,T];\bbR^{n})\to \bbR\) such that \(F=G\) on \([0,T]\times C([0,T];\bbR^{n})\) and \(G\in C^{1,2}_{b}\).
The vertical derivative of \(F\) is defined as the one of \(G\).
The above theorem ensures that such a definition is independent of the choice of \(G\).

\section{A convex duality}
\label{sec-4-dual}
In this section, we present a convex duality in Proposition~\ref{prop-4-dual}, which works in both discrete and continuous-time settings.
It decouples the loss function \(L\) and acts as an intermediate step to derive the duality for causal DRO problems.

\begin{asmp}
	\label{asmp-4-dual}
	We assume the following conditions:
	\begin{enumerate}[label=(\roman*)]
		\item   \(L:\bbR^{*}\to\bbR^{*}\) is a non-decreasing closed proper convex function with \(L(0)=0\) and \(L(+\infty)=+\infty\).
		\item \(c:\cX\times \cX\to\bbR^{*}\) is  lower semi-analytic and non-negative with \(c(x,y)=0\) if and only if \(x=y\) for any \(x,y\in\cX\).
		\item   \(f:\cX\to \bbR\) is upper semi-analytic and \(\mu\)-integrable.
	\end{enumerate}
\end{asmp}

\begin{prop}
	\label{prop-4-dual}
	Under Assumption~\ref{asmp-4-dual}, we have the duality
	\begin{equation*}
		V_{\c}=\sup_{\nu\in \scrP(\cX)}\{\E_{\nu}[f(X)]-L(\cT_{\c}(\mu,\nu))\}=\inf_{\lambda\geq 0}\{L^{*}(\lambda)+U(\lambda)\},
	\end{equation*}
	where \(L^{*}\) is the convex conjugate of \(L\), and \(U(\lambda)=\sup_{\pi\in\Pi_{\c}(\mu,*)}\E_{\pi}[f(Y)-\lambda c(X,Y)]\).
	Moreover, there exists a dual optimizer \(\lambda_{*}\) attaining the infimum.
	If there exist \(\nu_{*}\in\scrP(\cX)\) and \(\pi_{*}\in\Pi_{\c}(\mu,\nu_{*})\) such that
	\begin{align}
		\label{eqn-4-slack}
		L^{*}(\lambda_{*})=\lambda_{*}\E_{\pi_{*}}[c(X,Y)]-L(\E_{\pi_{*}}[c(X,Y)]) \text{ and } U(\lambda_{*})=\E_{\pi_{*}}[f(Y)-\lambda_{*}c(X,Y)],
	\end{align}
	then it holds
	\begin{equation*}
		V_{\c}=\E_{\nu_{*}}[f(X)]-L(\cT_{\c}(\mu,\nu_{*}))=L^{*}(\lambda_{*})+U(\lambda_{*}).
	\end{equation*}
\end{prop}

\begin{rmk}
	\label{rmk-4-ext}
	Let \(\cP\ni \mu\) be a convex subset of \(\scrP(\cX)\) and we write \(\Pi_{\c}(\mu,\cP)=\bigcup_{\nu\in\cP} \Pi_{\c}(\mu,\nu)\).
	One can extend the above duality to a closed proper concave functional \(F:\cP\to \bbR\).
	Under Assumption~\ref{asmp-4-dual} (i) and (ii), it holds
	\begin{equation*}
		\sup_{\nu\in\cP}\{F(\nu)-L(\cT_{\c}(\mu,\nu))\}=\inf_{\lambda\geq 0}\{L^{*}(\lambda)+\sup_{\nu\in \cP}\{F(\nu)-\cT_{\c}(\mu,\nu)\}\}.
	\end{equation*}
\end{rmk}

\begin{proof}
	\emph{Step 1.} We first consider the case \(L=+\infty\1_{(\delta,+\infty]}\) for some \(\delta\geq 0\).
	We write
	\begin{equation*}
		V(\delta)=\sup_{\cT_{\c}(\mu,\nu)\leq \delta} \E_{\nu}[f(X)].
	\end{equation*}
	If \(V(\delta)=+\infty\) for some \(\delta>0\), then there exists a sequence \(\{\pi_{n}\}_{n\geq 1}\) in \(\Pi_{\c}(\mu,*)\) such that \(\lim_{n\to\infty}\E_{\pi_{n}}[f(Y)]=+\infty\) and \(\E_{\pi_{n}}[c(X,Y)]\leq \delta\).
	In this case it is direct to verify the desired duality
	\begin{equation*}
		V(\delta)=  \inf_{\lambda\geq 0}\{\lambda\delta +\sup_{\pi\in\Pi_{\c}(\mu,*)}\{\E_{\pi}[f(Y)-\lambda c(X,Y)]\}\}=+\infty.
	\end{equation*}
	Now, we focus on the case \(V(\delta)<+\infty\) for any \(\delta\geq 0\).
	We claim that \(V\) is  proper, concave, and closed by naturally setting \(V\) as \(-\infty\) on the negative real line.
	The concavity of \(V\) follows directly from the convexity of the set of causal couplings \(\Pi_{\c}(\mu,*)\),	 see for example \citep[Theorems 7.6 and 7.7]{beiglbock20Denseness}.
	By Assumption~\ref{asmp-4-dual} (ii) and (iii), we obtain that \(V\) is proper and closed as \(V(0)=\E_{\mu}[f(X)]<\infty\).
	We calculate the concave conjugate of \(V\) as
	\begin{align*}
		V_{*}(\lambda) =\inf_{\delta\geq 0}\{\lambda\delta-\sup_{\cT_{\c}(\mu,\nu)\leq \delta}\E_{\nu}[f(X)]\} & =\inf_{\delta\geq 0}\inf_{\cT_{\c}(\mu,\nu)\leq \delta }\{\lambda\delta-\E_{\nu}[f(X)]\}      \\
		                                                                                                       & =\inf_{\nu\in \scrP(\cX)}\inf_{\delta\geq \cT_{\c}(\mu,\nu)}\{\lambda \delta-\E_{\nu}[f(X)]\} \\
		                                                                                                       & =\inf_{\nu\in \scrP(\cX)}\{\lambda  \cT_{\c}(\mu,\nu)-\E_{\nu}[f(X)]\}                        \\
		                                                                                                       & =\inf_{\pi\in\Pi_{\c}(\mu,*)}\E_{\pi}[\lambda c(X,Y)-f(Y)].
	\end{align*}
	By Fenchel--Moreau Theorem~\ref{thm-4-fenchel}, we have \(V=(V_{*})_{*}\) which yields
	\begin{align*}
		V(\delta) & =\inf_{\lambda\geq 0}\{\lambda\delta -V_{*}(\lambda)\}                                                 \\
		          & =\inf_{\lambda\geq 0}\{\lambda\delta -\inf_{\pi\in\Pi_{\c}(\mu,*)}\{\E_{\pi}[\lambda c(X,Y)-f(Y)]\}\}  \\
		          & =\inf_{\lambda\geq 0}\{\lambda\delta +\sup_{\pi\in\Pi_{\c}(\mu,*)}\{\E_{\pi}[f(Y)-\lambda c(X,Y)]\}\}.
	\end{align*}
	The desired duality follows by noting that \(L^{*}(\lambda)=\delta\lambda\) when \(L=+\infty\1_{(\delta,+\infty]}\).

	\emph{Step 2.}
	We consider a general penalization \(L\) satisfying Assumption~\ref{asmp-4-dual} (i).
	We notice that
	\begin{align*}
		\sup_{\nu\in\scrP(\cX)}\{\E_{\nu}[f(X)]-L(\cT_{\c}(\mu,\nu))\} & \leq \sup_{\nu\in\scrP(\cX)}\{V(\cT_{\c}(\mu,\nu))-L(\cT_{\c}(\mu,\nu))\} \\
		                                                               & \leq\sup_{\delta\geq 0}\{V(\delta)-L(\delta)\}.
	\end{align*}
	On the other hand, since \(L\) is non-decreasing, we have for any \(\delta\geq 0\)
	\begin{align*}
		V(\delta)-L(\delta)=\sup_{\cT_{\c}(\mu,\nu)\leq \delta} \{\E_{\nu}[f(X)]-L(\delta)\}\leq \sup_{\cT_{\c}(\mu,\nu)\leq \delta} \{\E_{\nu}[f(X)]-L(\cT_{\c}(\mu,\nu))\}.
	\end{align*}
	Hence, we derive \(\sup_{\nu\in\scrP(\cX)}\{\E_{\nu}[f(X)]-L(\cT_{\c}(\mu,\nu))\}=\sup_{\delta\geq 0}\{V(\delta)-L(\delta)\}\).
	Plugging the result from \emph{Step 1}, we obtain
	\begin{align*}
		\sup_{\nu\in\scrP(\cX)}\{\E_{\nu}[f(X)]-L(\cT_{\c}(\mu,\nu))\} & =\sup_{\delta\geq 0}\{V(\delta)-L(\delta)\}                                                                                        \\
		                                                               & =\sup_{\delta\geq 0}\inf_{\lambda\geq 0}\{\lambda \delta -L(\delta) + \sup_{\pi\in \Pi_{\c}(\mu,*)}\E_{\pi}[f(Y)-\lambda c(X,Y)]\} \\
		                                                               & :=\sup_{\delta\geq 0}\inf_{\lambda\geq 0}\{\lambda\delta-L(\delta)+ U(\lambda)\}.
	\end{align*}
	Here, \(U(\lambda)=\sup_{\pi\in \Pi_{\c}(\mu,*)}\E_{\pi}[f(Y)-\lambda c(X,Y)]\) either is equal to \(+\infty\) for any \(\lambda\geq 0\) or is a closed proper convex function of \(\lambda\).
	In the former case, it is direct to verify
	\begin{equation*}
		\sup_{\nu\in\scrP(\cX)}\{\E_{\nu}[f(X)]-L(\cT_{\c}(\mu,\nu))\}=\inf_{\lambda\geq 0}\{L^{*}(\lambda  )+U(\lambda)\}=+\infty.
	\end{equation*}
	In the latter case, the duality will follow from Lemma~\ref{lem-4-dual} below.

	\emph{Step 3.} We notice that by definition \((L^{*}+U)\) is lower semi-continuous and goes to infinity at infinity.
	Therefore, there exists a minimizer \(\lambda_{*}\).
	Given a pair \((\nu_{*},\pi_{*})\) satisfying the slackness condition \eqref{eqn-4-slack}, we derive
	\begin{align*}
		\E_{\nu_{*}}[f(X)]-L(\cT_{c}(\mu,\nu_{*})) & \geq \lambda_{*}\E_{\pi_{*}}[c(X,Y)]\!-\!L(\E_{\pi_{*}}[c(X,Y)])\!+\!\E_{\pi_{*}}[f(Y)\!-\!\lambda_{*}c(X,Y)] \\
		                                           & =L^{*}(\lambda_{*})+U(\lambda)=V_{\c}.
	\end{align*}
	The reverse direction is trivial and we complete the proof.
\end{proof}

We did not find a direct reference for the following lemma, so we include a proof for the completeness.
\begin{lem}
	\label{lem-4-dual}
	Let \(\varphi,\psi:\bbR\to \bbR^{*}\) be two closed proper convex functions.
	Then, we have the following  minimax theorem
	\begin{align*}
		\sup_{x\in\bbR}\inf_{y\in\bbR}\{xy-\varphi(x)+\psi(y)\}=\inf_{y\in\bbR}\sup_{x\in\bbR}\{xy-\varphi(x)+\psi(y)\}.
	\end{align*}
\end{lem}
\begin{rmk}
	This lemma generalizes Fenchel--Moreau theorem.
	By taking \(\psi\) as a linear function, we retrieve the classical Fenchel--Moreau theorem \(\varphi^{**}=\varphi\).
\end{rmk}

\begin{proof}
	By definition, we have
	\begin{align*}
		\sup_{x\in\bbR}\inf_{y\in\bbR}\{xy-\varphi(x)+\psi(y)\}\leq\inf_{y\in\bbR}\sup_{x\in\bbR}\{xy-\varphi(x)+\psi(y)\}.
	\end{align*}
	For the reverse direction, we notice that
	\begin{equation}
		\label{eqn-reverse}
		\inf_{y\in\bbR}\sup_{x\in\bbR}\{xy-\varphi(x)+\psi(y)\} =\inf_{y\in\bbR}\{\varphi^{*}(y)+\psi(y)\}.
	\end{equation}
	\emph{Case 1.} The above infimum is attained at \(y_{0}\).
	This implies that \(0\in \partial \varphi^{*}(y_{0})+\partial \psi(y_{0})\), and thus we may take \(x_{0}\in \partial \varphi^{*}(y_{0})\cap -\partial \psi(y_{0})\).
	By Proposition~\ref{prop-4-subdiff}, we have
	\begin{equation*}
		\varphi(x_{0})+\varphi^{*}(y_{0})=x_{0}y_{0},
	\end{equation*}
	and, for any \(y\in \mathbb{R}\),
	\begin{equation*}
		\psi(y)-\psi(y_{0})+x_{0}(y-y_{0})\geq 0.
	\end{equation*}
	Therefore, we derive
	\begin{align*}
		\sup_{x\in\bbR}\inf_{y\in\bbR}\{xy-\varphi(x_{0})+\psi(y)\} \! & \!\geq \inf_{y\in\bbR}\{x_{0}y-\varphi(x)+\psi(y)\}                                                    \\
		                                                               & \!=\inf_{y\in\bbR}\{ x_{0}y_{0}-\varphi(x_{0})  + \psi(y_{0}) + [\psi(y)-\psi(y_{0})+x_{0}(y-y_{0})]\} \\
		                                                               & \!\geq \varphi^{*}(y_0)+\psi(y_{0})\geq \inf_{y\in\bbR}\sup_{x\in\bbR}\{xy-\varphi(x)+\psi(y)\}.
	\end{align*}
	\emph{Case 2.} The infimum in \eqref{eqn-reverse} is not attained.
	To establish the reverse inequality, it suffices to consider the case \(\inf_{y\in\mathbb{R}}\{\varphi^{*}(y)+\psi(y)\}>-\infty\).
	As \((\varphi^{*}+\psi)\) is convex and does not attain its infimum, we deduce it is monotone.
	We may assume it is non-decreasing and obtain
	\begin{equation*}
		\inf_{y\in\mathbb{R}}\{\varphi^{*}(y)+\psi(y)\}= \lim_{y\to-\infty}\{\varphi^{*}(y)+\psi(y)\}>-\infty.
	\end{equation*}
	In particular, we have \(\lim_{y\to -\infty} \partial \varphi^{*} (y)+ \partial \psi (y)= \{0\}\), where the convergence is interpreted as subsets of \(\mathbb{R}\).
	Since \(\varphi^{*}\) and \(\psi\) are both convex, their sub-differentials are non-decreasing.
	This yields the existence of \(x_{0}\) such that
	\[
		\lim_{y\to -\infty}\partial \varphi^{*}(y)=-\lim_{y\to -\infty}\partial \psi(y)= x_{0}.
	\]
	We note that \(\varphi^{*}(y)-x_{0}y\) and \(\psi(y)+x_{0}y\) are both non-decreasing.
	Therefore, we derive
	\begin{align*}
		\inf_{y\in\bbR}\{\varphi^{*}(y)+\psi(y)\} & =\inf_{y\in\bbR}\{\varphi^{*}(y)-x_{0}y\}+\inf_{y\in\bbR}\{\psi(y)+x_{0}y\} \\
		                                          & =-\varphi(x_{0})+\inf_{y\in\bbR}\{\psi(y)+x_{0}y\}                          \\
		                                          & \leq \sup_{x\in\bbR}\inf_{y\in\bbR}\{xy-\varphi(x)+\psi(y)\}.
	\end{align*}
\end{proof}

\section{Discrete-time results}
\label{sec-4-disc}

In this section, we focus on the discrete-time setting.
Let \(I=\{0,1,\dots,N\}\) and \(\cX=\cX_{0}\times \cX_{1}\times \cdots\times \cX_{N}\).
In light of Proposition~\ref{prop-4-dual}, it suffices to compute \(\sup_{\pi\in\Pi_{\c}(\mu,*)}\E_{\pi}[f(Y)-\lambda c(X,Y)]\), which can be viewed as a causal optimal transport problem  constrained only by the source measure.
We will exploit a key dynamic temporal structure of causal couplings from the following proposition, first established in \citet{backhoff-veraguas17Causal}.
\begin{prop}
	\label{prop-4-couple}
	The following statements are equivalent:
	\begin{itemize}
		\item \(\pi\in\Pi(\mu,*)\) is a causal coupling with the first marginal \(\mu\).
		\item Decomposing \(\pi\) in terms of successive regular kernels
		      \begin{equation*}
			      \pi(\md x,\md y)=\pi_{0}(\md x_{0},\md y_{0})\pi_{x_{0},y_{0}}(\md x_{1},\md y_{1})\cdots \pi_{x_{0:N-1},y_{0:N-1}}( \md x_{N},\md y_{N}),
		      \end{equation*}
		      for any \(1\leq n\leq N \) and  \(\pi\)-almost surely \(x_{0:n-1}\), \(y_{0:n-1}\) we have
		      \begin{equation*}
			      \pi_{x_{0:n-1},y_{0:n-1}}(\md x_{n},\md y_{n})\in\Pi(\mu_{x_{0:n-1}}(\md x_{n}),*),
		      \end{equation*}
		      and \(\pi_{0}(\md x_{0},\md y_{0})\in \Pi(\mu_{0}(\md x_{0}),*)\).
	\end{itemize}
\end{prop}

\begin{thm}
	\label{thm-4-disc}
	Under Assumption~\ref{asmp-4-dual}, we have a dynamic duality formula
	\begin{equation}
		\label{eqn-d}
		V_{\c}=\sup_{\nu\in\scrP(\cX)}\{\E_{\nu}[f(X)]-L(\cT_{\c}(\mu,\nu))\}=\inf_{\lambda\geq 0}\{L^{*}(\lambda)+U(\lambda)\},
	\end{equation}
	where \(U\) is given by
	\begin{equation*}
		U(\lambda)=\int_{\mathcal{X}_{0}}\sup_{y_{0}\in\cX_{0}}\biggl\{ \cdots\int_{\mathcal{X}_{N}}\sup_{y_{N}\in\cX_{N}}\{f(y)-\lambda c(x,y)\}\mu_{x_{0:N-1}}(\md x_{N})\cdots\biggr\} \mu_{0}(\md x_{0}),
	\end{equation*}
	or, equivalently,
	\begin{equation}
		\label{eqn-ddual}
		U(\lambda)=\mu_{0}(\sup_{y_{0}\in\cX_{0}}\{ \cdots\mu_{x_{0:N-1}}(\sup_{y_{N}\in\cX_{N}}\{f(y)-\lambda c(x,y)\})\cdots\}).
	\end{equation}
\end{thm}
\begin{rmk}
	If \(N=0\), we retrieve the classical Wasserstein DRO duality results in \citet{blanchet19Quantifying,bartl20Computational,zhang24Short} for this static setting.
	Moreover, if  \(f\) and \(c\) are time-separable, i.e., \(f(x)=\sum_{n=0}^{N}f_{n}(x_{n})\) and \(c(x,y)=\sum_{n=0}^{N}c_{n}(x,y)\), then we have
	\[
		U(\lambda)=\sup_{y\in \mathcal{Y}}\E_{\mu}[f(y)-\lambda c(x,y)].
	\]
	This implies that the classical Wasserstein DRO coincides with the causal DRO.
\end{rmk}

\begin{proof}

	We first show that the following quantities are well-defined:
	\[U_{N}(x_{0:N-1},y_{0:N-1}):=\!\!\!\!\sup_{\pi_{N}\in\Pi(\mu_{x_{0:N-1}},*)}\!\!\!\!\E_{\pi_{N}}[f(y_{0:N-1},Y_{N})-\lambda c(x_{0:N-1},X_{N},y_{0:N-1},Y_N)],\]
	and for \(1\leq n\leq N-1\)
	\begin{equation*}
		U_{n}(x_{0:n-1},y_{0:n-1}):=\sup_{\pi_{n}\in\Pi(\mu_{x_{0:n-1}},*)}\E_{\pi_{n}}[U_{n+1}(x_{0:n-1},X_{n},y_{0:n-1},Y_{n})].
	\end{equation*}
	We claim that \(U_{N}\) is upper semi-analytic.
	It follows from \citet[Proposition 7.48]{bertsekas96Stochastic} that
	\begin{equation*}
		D=\{(x_{0:N-1},y_{0:N-1},\pi):\,(x,y)\in\cX\times\cX,\,\pi\in \Pi(\mu_{x_{0:N-1}},*)\}
	\end{equation*}
	is Borel, and
	\begin{equation*}
		(x_{0:N-1},y_{0:N-1},\pi) \mapsto \E_{\pi}[f(Y)-\lambda c(X,Y)]
	\end{equation*}
	is upper semi-analytic.
	By \citet[Proposition 7.47]{bertsekas96Stochastic}, we obtain that \(U_{N}\) is upper semi-analytic.
	Therefore, \(U_{N}\) is universally measurable, and this implies that \(U_{N-1}\) is well-defined.
	Recursively, we can show that \(U_{n}\) is again upper semi-analytic for any \(1\leq n\leq N\).
	By Proposition~\ref{prop-4-couple}, we decompose the optimization problem \(\sup_{\pi\in\Pi_{\c}(\mu,*)}\E_{\pi}[f(Y)-\lambda c(X,Y)]\) into single step problems as
	\begin{align*}
		U(\lambda) & =\sup_{\pi\in\Pi_{\c}(\mu,*)}\E_{\pi}[f(Y)-\lambda c(X,Y)]                                                                                    \\
		           & =\sup_{\pi_{0}\in\Pi(\mu_{0},*)}\E_{\pi_{0}}\Bigl[\cdots\sup_{\pi_{N}\in\Pi(\mu_{x_{1:N-1}},*)}\E_{\pi_{N}}[f(Y)-\lambda c(X,Y)]\cdots\Bigr].
	\end{align*}
	As  we have shown that \(U_{n}\) is upper semi-analytic for any \(1\leq n \leq N\), applying Lemma~\ref{lem-static} we derive \eqref{eqn-ddual}.
\end{proof}

Our next result answers the question when the bi-causal and the causal ambiguity sets are equivalent in terms of their corresponding distributional risks.
\begin{asmp}
	\label{asmp-4-reg}
	We assume there exists \(p\geq 1\) such that the following conditions hold:
	\begin{enumerate}[label=(\roman*)]
		\item \(L:\bbR^{*}\to\bbR^{*}\) is non-decreasing and continuous on its domain \(\{L<\infty\}\).
		\item \(c:\cX\times\cX\to\bbR\) is continuous and has a polynomial growth
		      \[|c(x,y)|\leq C(1+d_{\cX}(\tilde{x},x)^{p}+d_{\cX}(\tilde{x},y)^{p})\quad \text{ for some } \tilde{x}\in\cX.\]
		\item \(f:\cX\to \bbR\) is continuous and has a polynomial growth \[|f(x)|\leq C(1+d_{\cX}(\tilde{x},x)^{p})\quad \text{ for some } \tilde{x}\in\cX.\]
	\end{enumerate}
\end{asmp}

\begin{thm}
	\label{thm-4-reg}
	We assume that \(\cX\) has no isolated points.
	Under Assumption~\ref{asmp-4-reg}, we have \(V_{\c}=V_{\bc}\), i.e.,
	\begin{equation*}
		\sup_{\nu\in\scrP(\cX)}\{\E_{\nu}[f(X)]-L(\cT_{\c}(\mu,\nu))\}=\sup_{\nu\in\scrP(\cX)}\{\E_{\nu}[f(X)]-L(\cT_{\bc}(\mu,\nu))\}.
	\end{equation*}
\end{thm}
\begin{proof}
	We lift the problem to the space of adapted stochastic processes introduced in Section \ref{sec-pre}.
	Recall by \(\AP\) we denote the space of adapted stochastic processes with paths in \(\cX\), and by \(\NP\) we denote the space of naturally filtered stochastic processes.
	It is clear that we can reformulate the desired identity as
	\begin{equation*}
		\sup_{\bbY\in\NP}\{\E_{\Q}[f(Y)]-L(\cT_{\c}(\bbX,\bbY))\}=\sup_{\bbY\in\NP}\{\E_{\Q}[f(Y)]-L(\cT_{\bc}(\bbX,\bbY))\},
	\end{equation*}
	where \(\bbX=(\cX,\cF,\bF,\mu,X)\) and \(\bbY=(\Omega^{\bbY},\cF^{\bbY},\bF^{\bbY},\Q,Y)\).
	Since we assume \(\cX\) has no isolated points, by \citet[Theorem 5.4]{bartl24Wasserstein} \(\NP\) is a dense subset of \(\AP\) in \(\AW_{p}\).
	By Assumption~\ref{asmp-4-reg}  (ii) and \citet[Theorem 3.6]{eckstein24Computational}, both \(\bbY\mapsto\cT_{\c}(\bbX,\bbY)\) and  \(\bbY\mapsto\cT_{\bc}(\bbX,\bbY)\)  are continuous with respect to \(\AW_{p}\).
	Combining these facts with Assumption~\ref{asmp-4-reg} (iii), we derive
	\begin{equation*}
		\sup_{\bbY\in\NP}\{\E_{\Q}[f(Y)]-L(\cT_{\c}(\bbX,\bbY))\}= \sup_{\bbY\in\AP}\{\E_{\Q}[f(Y)]-L(\cT_{\c}(\bbX,\bbY))\},
	\end{equation*}
	and
	\begin{equation*}
		\sup_{\bbY\in\NP}\{\E_{\Q}[f(Y)]-L(\cT_{\bc}(\bbX,\bbY))\}= \sup_{\bbY\in\AP}\{\E_{\Q}[f(Y)]-L(\cT_{\bc}(\bbX,\bbY))\}.
	\end{equation*}
	For any \(\bbY\in \AP\), we construct \(\widetilde{\bbY}=(\Omega^{\bbY},\cF^{\bbY},\widetilde{\bF}^{\bbY},Q,Y)\in \AP\) where \(\widetilde{\bF}^{\bbY}=\{\cF^{\bbY}\}_{t\in I}\) is the constant filtration with the richest \(\sigma\)-algebra.
	This yields the inclusion \(\Pi_{\c}(\bbX,\bbY)\subseteq \Pi_{\bc}(\bbX,\widetilde{\bbY})\) as any coupling is causal from \(\widetilde{\bbY}\) to \(\bbX\).
	Hence, we deduce \(\cT_{\c}(\bbX,\bbY)\geq \cT_{\bc}(\bbX,\widetilde{\bbY})\), and together with Assumption~\ref{asmp-4-reg} (i) we show
	\begin{align*}
		\sup_{\bbY\in\AP}\{\E_{\Q}[f(Y)]-L(\cT_{\c}(\bbX,\bbY))\} & =\sup_{\bbY\in\NP}\{\E_{\Q}[f(Y)]-L(\cT_{\c}(\bbX,\bbY))\}                  \\
		                                                          & \leq \sup_{\bbY\in\NP}\{\E_{\Q}[f(Y)]-L(\cT_{\bc}(\bbX,\widetilde{\bbY}))\} \\
		                                                          & \leq \sup_{\bbY\in\AP}\{\E_{\Q}[f(Y)]-L(\cT_{\bc}(\bbX,\bbY))\}.
	\end{align*}
	The reverse direction is trivial, and we conclude the proof.
\end{proof}

We remark that this result generalizes the arguments in \citet[Lemma 3.1]{bartl23Sensitivity} to abstract Polish spaces.

\section{Continuous-time results}
\label{sec-4-cont}
We set up the continuous-time problem as follows.
Let \((\Omega,\cF,P)\) be a Polish probability space supporting a standard \(n\)-dimensional Brownian motion \(B\).
We consider a path-dependent SDE given by
\begin{equation}
	\label{eqn-4-nsde}
	\alpha(t)=\int_{0}^{t}b(s,\alpha)\md s +\int_{0}^{t}\sigma(s, \alpha)\md B(s),
\end{equation}
where \(b\) and \(\sigma\) are non-anticipative as in Definition \ref{def-nonanti}.
\begin{asmp}
	\label{asmp-4-sde}
	We assume  that SDE \eqref{eqn-4-nsde} has a unique strong solution \(\alpha\), and that the law of \(\alpha(t)\) is non-atomic for any \(0<t\leq T\).
\end{asmp}
The strong well-posedness holds, for example in \citet[Theorem 8.1]{bally16Stochastic}, if the coefficients \(b\) and \(\sigma\) are Lipschitz continuous and have linear growth.
In a Markovian setting, the celebrated H\"ormander theorem gives a sufficient condition under which the law of \(\alpha(t)\) is non-atomic.
More recently, it has been extended to a path-dependent SDE setting in \citet{ohashi21Smoothness}.

We take the reference model as \(\mu= \alpha_{\#}P\) and  are interested in the case where \(c(x,y)=\|x-y\|_{H^{1}_{0}}^{2}\) as the square of the Cameron--Martin norm on the path space \(\cX=C_{0}([0,T];\bbR^{n})\).

\begin{asmp}
	\label{asmp-4-cont-f}
	We assume the following conditions:
	\begin{enumerate}[label=(\roman*)]
		\item   \(L:\bbR^{*}\to\bbR^{*}\) is a non-decreasing closed proper convex function with \(L(0)=0\) and \(L(+\infty)=+\infty\).
		\item \(f:\cX\to\bbR\) is continuous and has a quadratic growth \(|f(x)|\leq C(1+\|x\|_{\infty}^{2})\).
	\end{enumerate}
\end{asmp}

\begin{thm}
	\label{thm-4-cont}
	Assume that  the semi-linear path-dependent PDE
	\begin{equation}
		\label{eqn-4-ppde}
		\begin{aligned}
			\cD & U (t, x, y ;\lambda)  + b(t,x)^{\intercal}(\nabla_{x}+\nabla_{y})U(t,x,y;\lambda)                                                                                                \\
			    & + \frac{1}{2} \tr(\sigma(t,x)^{\intercal}\sigma(t,x)(\nabla_{xx}^{2}+\nabla_{yy}^{2}+2\nabla_{xy}^{2})U(t,x,y;\lambda)) +\frac{1}{4\lambda} \|\nabla_{y}U(t,x,y;\lambda)\|^{2}=0
		\end{aligned}
	\end{equation}
	with boundary condition \(U(T,x,y;\lambda)=f(y)\) has a non-anticipative \(C_{b}^{1,2}\) solution \(U\) such that
	\(\nabla_{y}U(t,x,y; \textcolor{red}{\lambda})\) is Lipschitz in \((x,y)\) with a linear growth.
	Then under Assumptions~\ref{asmp-4-sde} and~\ref{asmp-4-cont-f} we have
	\begin{equation*}
		V_{\c}=\sup_{\nu\in \scrP(\cX)}\{\E_{\nu}[f(Y)]-L(\cT_{\c}(\mu,\nu))\}=\inf_{\lambda\geq 0}\{L^{*}(\lambda)+U(\lambda)\},
	\end{equation*}
	where
	\begin{equation*}
		U(\lambda):=U(0,0,0;\lambda)=\sup_{\pi\in \Pi_{\c}(\mu,*)}\E_{\pi}[f(Y)-\lambda c(X,Y)].
	\end{equation*}
\end{thm}

\begin{rmk}
	The derivatives in Equation \eqref{eqn-4-ppde} are interpreted as the horizontal and the vertical derivatives introduced in Definitions \ref{defn-hori} and \ref{defn-vert}.
	In particular, the horizontal derivative is given by
	\begin{equation*}
		\mathcal{D}U(t,x,y;\lambda)=\lim_{h\to 0^{+}}\frac{U(t+h, x(\cdot\wedge t), y(\cdot\wedge t); \lambda)-U(t, x(\cdot\wedge t), y(\cdot\wedge t); \lambda)}{h}.
	\end{equation*}
\end{rmk}

\begin{rmk}
	When the reference model \(\mu\) is the Wiener measure \(\gamma\), the above path-dependent PDE \eqref{eqn-4-ppde} reads as
	\begin{equation*}
		\cD U (t,x,y;\lambda) + \frac{1}{2}\tr((\nabla_{xx}^{2}+\nabla_{yy}^{2}+2\nabla_{xy})U(t,x,y;\lambda)) +\frac{1}{4\lambda} \|\nabla_{y}U(t,x,y;\lambda)\|^{2}=0.
	\end{equation*}
	The change of  variables \(U=2\lambda \log (W)\) turns the above equation to a linear PDE and yields a unique classical solution given explicitly by
	\begin{equation*}
		U(t,x,y;\lambda)= 2\lambda \log \lb( \int_{C([0,T];\bbR^{n})}\exp\lb(\frac{f(z)}{2\lambda}\rb) \gamma_{y(\cdot\wedge t)}(\md z)\rb).
	\end{equation*}
\end{rmk}

The following result is adapted from \citet[Theorem 1.2]{beiglbock20Denseness}  which shows  that the set of Monge causal couplings with a first marginal \(\mu\) is dense in \(\Pi_{\c}(\mu,*)\).
We postpone the proof to the end of the current section.

\begin{lem}
	\label{lem-4-dense-monge}
	Let \(\Gamma_{\c}(\mu,*)\) denote the set of  Monge causal couplings with a fixed first marginal \(\mu\).
	Under Assumption~\ref{asmp-4-sde}, we have
	\begin{equation*}
		\sup_{\pi\in\Pi_{\c}(\mu,*)}\E_{\pi}[f(Y)-\lambda c(X,Y)]=\sup_{\pi\in\Gamma_{\c}(\mu,*)}\E_{\pi}[f(Y)-\lambda c(X,Y)].
	\end{equation*}
\end{lem}

\begin{proof}[Proof of Theorem~\ref{thm-4-cont}]
	\emph{Step 1.} We first show the solution to \eqref{eqn-4-ppde} coincides with the value function of a non-Markovian stochastic optimal control problem.
	To simplify the notation, we fix \(\lambda>0\) and omit the \(\lambda\) argument in \(U\).
	For two paths \(\omega,\eta\in C_{0}([0,T];\bbR^{n})\) with \(\omega(s)=\eta(s)\) we introduce their concatenation at time \(s\) as
	\begin{equation*}
		\omega\otimes_{s} \eta (t):= \begin{cases}
			\omega(t), & \text{if } t\leq s, \\
			\eta(t),   & \text{if } t>s.
		\end{cases}
	\end{equation*}
	Let \((\alpha^{s,x},\beta^{s,x,y,u})\) be a controlled system given by
	\begin{equation*}
		\left\{
		\begin{aligned}
			\alpha^{s,x}(\cdot)    & = x(s)+\int_{s}^{\cdot}b(t, x\otimes_{s}\alpha^{s,x})\md t +\int_{s}^{\cdot}\sigma(t, x\otimes_{s}\alpha^{s,x})\md B(t),        \\
			\beta^{s,x,y,u}(\cdot) & = y(s)+ \int_{s}^{\cdot}[b(t,x\otimes_{s}\alpha^{s,x})+u(t)]\md t +\int_{s}^{\cdot}\sigma(t, x\otimes_{s}\alpha^{s,x})\md B(t),
		\end{aligned}\right.
	\end{equation*}
	with the aim of maximizing the objective
	\begin{equation*}
		J(s,x,y, u)=\E_{P}\lb[f(y\otimes_{s} \beta^{s,x,y,u}) -\lambda \int_{s}^{T}\|u(t)\|^{2}\md t\rb]
	\end{equation*}
	over  \( \ad([s,T])\) the set of \(\bF^{B}\)-progressively measurable processes.
	The \emph{Hamiltonian} associated to the control problem \(\cH:[0,T]\times C([0,T];\bbR^{2n})\times (\bbR^{n})^{2}\times (\bbR^{n\times n})^{3}\) is given by
	\begin{align*}
		 & \qquad \cH(t,x,y,\rho_{x},\rho_{y},A_{xx},A_{yy},A_{xy})                                                                                                                                    \\
		 & = \!\sup_{u\in\bbR^{n}}\lb\{ b(t,x)^{\intercal}(\rho_{x}+\rho_{y})+\frac{1}{2}\tr(\sigma(t,x)^{\intercal}\sigma(t,x)(A_{xx}+A_{yy}+2A_{xy}))+ u^{\intercal}\rho_{y} -\lambda \|u\|^{2}\rb\} \\
		 & =b(t,x)^{\intercal}(\rho_{x}+\rho_{y})+\frac{1}{2}\tr(\sigma(t,x)^{\intercal}\sigma(t,x)(A_{xx}+A_{yy}+2A_{xy})) +\frac{\|\rho_{y}\|^{2}}{4\lambda}.
	\end{align*}
	In particular, the above supremum is attained when \(u= \frac{1}{2\lambda}\rho_{y}\).
	Since we assume \(U\) is a \(C^{1,2}_{b}\) solution, by \citet[Theorem 8.16]{bally16Stochastic}, it is clear that
	\begin{equation*}
		U(s,x,y)\geq \sup_{u\in \ad([s,T])} J(s,x,y,u).
	\end{equation*}
	On the other hand, as we assume \(\nabla_{y}U\) is Lipschitz, there exists a unique solution to
	\begin{equation}
		\label{eqn-4-optimal}
		\left\{\begin{aligned}
			\alpha^{s,x}(\cdot)      & = x(s)+\int_{s}^{\cdot}b(t, x\otimes_{s}\alpha^{s,x})\md t +\int_{s}^{\cdot}\sigma(t, x\otimes_{s}\alpha^{s,x})\md B(t),                                          \\
			\beta^{s,x,y}_{*}(\cdot) & = y(s)+ \int_{s}^{\cdot}\lb[b(t,x\otimes_{s}\alpha^{s,x})+ \frac{1}{2\lambda}\nabla_{y} U(t,x\otimes_{s} \alpha^{s,x} , y \otimes_{s} \beta^{s,x,y}_{*})\rb]\md t \\
			                         & \hspace{7cm} +\int_{s}^{\cdot}\sigma(t, x\otimes_{s}\alpha^{s,x})\md B(t).
		\end{aligned}\right.
	\end{equation}
	This yields that \(u_{*}(t)= \frac{1}{2\lambda}\nabla_{y} U(t,x\otimes_{s} \alpha^{s,x} , y \otimes_{s} \beta^{s,x,y}_{*})\in \ad([s,T])\) is an optimal control since \(u=\frac{1}{2\lambda}\rho_{y}\) attained the supremum in the Hamiltonian \(\cH\).
	Therefore, we derive
	\begin{equation*}
		U(s,x,y)= J(s,x,y,u_{*})=\sup_{u\in\ad([s,T])}J(s,x,y,u).
	\end{equation*}

	\emph{Step 2.} We show that  \(U(0,0,0)\leq \sup_{\pi\in \Pi_{\c}(\mu,*)}\E_{\pi}[f(Y)-\lambda c(X,Y)]\).
	From the previous step, we obtain
	\begin{align*}
		U(0,0,0)=\E_{\pi^{u_{*}}}[f(Y)-\lambda c(X,Y)],
	\end{align*}
	where \(\pi^{u_{*}}=(\alpha,\beta_{*})_{\#}\P\).
	It suffices to show that \(\pi^{u_{*}}\in \Pi_{\c}(\mu,*)\).
	We notice that  \((\alpha,B)_{\#}\P\) is a bi-causal coupling by \citet[Theorem 3.2]{cont24Causal}.
	Moreover, \((B,u_{*})_{\#}\P\) is a causal coupling  since \(u_{*}\) is \(\bF^{B}\)-progressively measurable.
	Hence, by the gluing lemma of \citet[Lemma 3.1]{eckstein24Computational}, we derive  that \((\alpha,u_{*})_{\#}\P\) is a causal coupling.
	Combining this with the fact that \(\md \beta_{*}=\md \alpha+ \md u_{*}\), we deduce \(\pi^{u_{*}}\in \Pi_{\c}(\mu,*)\).

	\emph{Step 3.} We establish that \(U(0,0,0)=\sup_{\pi\in \Pi_{\c}(\mu,*)}\E_{\pi}[f(Y)-c(X,Y)]\).
	For the reverse direction, let \(\widehat{\ad}([0,T])\) be the set of \(\bF^{\alpha}\)-progressively measurable processes.
	We remark that \(\widehat{\ad}([0,T])\subseteq \ad([0,T])\) as \(\alpha\) is \(\mathbf{F}^{B}\)-adapted.
	Recall \(\Gamma_{\c}(\mu,*)\) is the set of causal Monge couplings with a fixed first marginal \(\mu\).
	For any \(\pi\in \Gamma_{\c}(\mu,*)\) with \(\E_{\pi}[c(X,Y)]<\infty\), there exists a non-anticipative map \(\Phi:[0,T]\times C([0,T];\bbR^{n})\to \bbR^{n}\) such that \(Y(t)= X(t) + \int_{0}^{t} \Phi(s,X)\md s\) holds \(\pi\)-a.s.
	This induces a control \(u^{\pi}(t):= \Phi(t,\alpha)\in \widehat{\ad}([0,T])\).
	Therefore, together with Lemma~\ref{lem-4-dense-monge} we have
	\begin{align*}
		U(0,0,0)=\sup_{u\in \ad([0,T])}J(0,0,0,u) & \geq \sup_{u\in \widehat{\ad}([0,T])}J(0,0,0,u)                     \\
		                                          & \geq  \sup_{\pi\in \Gamma_{\c}(\mu,*)}\E_{\pi}[f(Y)-\lambda c(X,Y)] \\
		                                          & = \sup_{\pi\in \Pi_{\c}(\mu,*)}\E_{\pi}[f(Y)-\lambda c(X,Y)].
	\end{align*}
	Since Assumption~\ref{asmp-4-cont-f} implies Assumption~\ref{asmp-4-dual},
	applying Proposition~\ref{prop-4-dual} we conclude the proof.
\end{proof}

\begin{prop}
	Under the conditions of Theorem~\ref{thm-4-cont}, if further \(U(\lambda)\) is differentiable on its domain, then we have \(V_{\c}=V_{\bc}\), i.e.,
	\begin{equation*}
		\sup_{\nu\in \scrP(\cX)}\{\E_{\nu}[f(Y)]-L(\cT_{\c}(\mu,\nu))\}=\sup_{\nu\in \scrP(\cX)}\{\E_{\nu}[f(Y)]-L(\cT_{\bc}(\mu,\nu))\}.
	\end{equation*}
\end{prop}
\begin{proof}
	\emph{Step 1.}  We recall that \(\pi_{*}=(\alpha,\beta_{*})_{\#}P\) attains the supremum in \[U(\lambda)=\sup_{\pi\in \Pi_{\c}(\mu,*)}\E_{\pi}[f(Y)-\lambda c(X,Y)]\] where \((\alpha,\beta_{*})\) is given in \eqref{eqn-4-optimal}.
	We claim that \(\pi_{*}\) is actually a bi-causal coupling.
	To see this, we can reformulate \eqref{eqn-4-optimal} as
	\begin{equation*}
		\alpha(t) =-\int_{0}^{t}\frac{1}{2\lambda}\nabla_{y} U(s,\alpha,  \beta_{*})\md s  + \beta_{*}(t).
	\end{equation*}
	As we assume  \(\nabla_{y}U\) is Lipschitz, there exists a non-anticipative functional \(F\) such that \(\alpha(t)=F(t,\beta_{*})\) holds \(\pi_{*}\)-a.s.
	This implies that \(\pi_{*}\) is  causal from \(\beta_{*}\) to \(\alpha\), and hence it is a bi-causal coupling.

	\emph{Step 2.}  Following Proposition~\ref{prop-4-dual},
	there exists a minimizer \(\lambda_{*}\)  of \(\inf_{\lambda\geq 0}\{L^{*}(\lambda)+U(\lambda)\}\).
	This implies \(0\in \partial L^{*}(\lambda_{*})+\partial U(\lambda_{*})\).
	We notice that 	\begin{align*}
		U(\lambda)-U(\lambda_{*}) & \geq  \E_{\pi_{*}}[f(Y)-\lambda c(X,Y)]-\E_{\pi_{*}}[f(Y)-\lambda_{*} c(X,Y)] \\
		                          & = -\E_{\pi_{*}}[c(X,Y)] (\lambda-\lambda_{*}).
	\end{align*}
	Together with the assumption that \(U\) is differentiable, we obtain \(\{-\E_{\pi_{*}}[c(X,Y)]\} = \partial U(\lambda_{*})\).
	Hence, we deduce that \(\E_{\pi_{*}}[c(X,Y)]\in \partial L^{*}(\lambda_{*})\).
	Therefore, by Proposition~\ref{prop-4-subdiff} we obtain
	\begin{equation*}
		L^{*}(\lambda_{*})=\lambda_{*}\E_{\pi_{*}}[c(X,Y)]-L(\E_{\pi_{*}}[c(X,Y)]).
	\end{equation*}
	It follows from the bi-causality of \(\pi_{*}\) and Proposition~\ref{prop-4-dual} that
	\begin{equation*}
		V_{\c}=\E_{\pi_{*}}[f(Y)]-L(\E_{\pi_{*}}[c(X,Y)])\leq \sup_{\nu\in \scrP(\cX)}\{\E_{\nu}[f(Y)]-L(\cT_{\bc}(\mu,\nu))\}=V_{\bc}.
	\end{equation*}
	The reverse direction is trivial as \(\Pi_{\bc}(\mu,*)\subseteq\Pi_{\c}(\mu,*)\).
\end{proof}

\begin{proof}[Proof of Lemma~\ref{lem-4-dense-monge}]
	When \(\lambda=0\), it is clear that the identity holds with both sides equal to \(\sup_{x\in\cX}f(x)\).
	For \(\lambda>0\), we rely on the proof of \citet[Theorem 5.2]{beiglbock20Denseness}.
	Let \(Z=Y-X\).
	We fix a \(\pi_{0}\in \Pi_{\c}(\mu,*)\) such that \(\E_{\pi_{0}}[\|Z\|^{2}_{H^{1}_{0}}]<\infty\).
	We notice \(\tilde{\pi}=(X,Z)_{\#}\pi_{0}\) is also in \(\Pi_{\c}(\mu,*)\).
	By \citet[Theorem 1.2]{beiglbock20Denseness}, there exist piecewise linear \(\bF^{X}\)-adapted processes \(\{Z_{n}\}_{n\geq 1}\) such that  \((X,Z_{n})_{\#}\pi_{0}\) converges weakly to \(\tilde{\pi}\).
	In particularw, \(Z_{n}\) can be further chosen such that \(\lim_{n\to\infty}\E_{\pi_{0} }[\|Z_{n}\|^{2}_{H^{1}_{0}}]=\E_{\pi_{0}}[\|Z\|^{2}_{H^{1}_{0}}].\)
	Therefore, \(\pi_{n}=(X,X+Z_{n})_{\#}\pi_{0}\in \Gamma_{\c}(\mu,*)\) and satisfies
	\begin{equation*}
		\sup_{\pi\in\Gamma_{\c}(\mu,*)}\E_{\pi}[f(Y)-\lambda c(X,Y)] \geq \lim_{n\to\infty}\E_{\pi_{n}}[f(Y)-\lambda c(X,Y)]= \E_{\pi_{0}}[f(Y)-\lambda c(X,Y)].
	\end{equation*}
	The arbitrary choice of \(\pi_{0}\) allows us to derive
	\begin{equation*}
		\sup_{\pi\in\Gamma_{\c}(\mu,*)}\E_{\pi}[f(Y)-\lambda c(X,Y)] \geq  \sup_{\pi\in\Pi_{\c}(\mu,*)}\E_{\pi}[f(Y)-\lambda c(X,Y)].
	\end{equation*}
	The reverse direction is trivial as \(\Gamma_{\c}(\mu,*)\subseteq \Pi_{\c}(\mu,*)\).
\end{proof}

\section{Distributionally robust expected shortfall}
\label{sec-4-es}
In this section, we consider the extension of our duality formula from a linear expectation to a convex risk measure in the discrete-time setting.
Let \(f:\cX\to\bbR\) be a path-dependent payoff of a contingent claim  and \(\alpha\in(0,1)\) a risk aversion parameter.
The expected shortfall  \(\av\) is a law-invariant convex risk measure given by
\begin{equation*}
	\av(f;\mu):=\sup_{\eta\ll \mu,\, \frac{\md \eta}{\md \mu}\leq \alpha^{-1}}\E_{\eta}[f(X)].
\end{equation*}
We are interested in its performance under model misspecification, and write its causal distributionally robust counterpart  \(\bav\) as
\begin{equation*}
	\bav(f;\mu):=\sup_{\nu\in\scrP(\cX)}\{\av(f;\nu)-L(\cT_{\c}(\mu,\nu))\}.
\end{equation*}
\begin{asmp}
	\label{asmp-4-es}
	We assume the following conditions:
	\begin{enumerate}[label=(\roman*)]
		\item   \(L:\bbR^{*}\to\bbR^{*}\) is a non-decreasing closed proper convex function with \(L(0)=0\) and \(L(+\infty)=+\infty\).
		\item \(c:\cX\times \cX\to\bbR^{*}\) is  lower semi-continuous and non-negative with \(c(x,y)=0\) if and only if \(x=y\) for any \(x,y\in\cX\).
		\item   \(f:\cX\to \bbR\) is non-negative and upper semi-continuous.
	\end{enumerate}
\end{asmp}

We derive the following duality representation of \(\bav(f)\).

\begin{thm}
	\label{thm-avar}
	Under Assumption~\ref{asmp-4-es}, we have
	\begin{align*}
		\bav(f;\mu) & =\inf_{\lambda,\gamma\geq 0}\{L^{*}(\lambda)+\gamma + U(\lambda,\gamma)\},
	\end{align*}
	where \(U(\lambda,\gamma)=\sup_{\pi\in\Pi_{\c}(\mu,*)}\E_{\pi}[\alpha^{-1}(f(Y)-\gamma)^{+}-\lambda c(X,Y)]\) can be computed by
	\begin{equation*}
		U(\lambda,\gamma)=\mu_{0}(\sup_{y_{0}\in\cX_{0}}\{ \cdots\mu_{x_{0:N-1}}(\sup_{y_{N}\in\cX_{N}}\{\alpha^{-1}(f(y)-\gamma)^{+}-\lambda c(x,y)\})\cdots\}).
	\end{equation*}
\end{thm}

\begin{proof}
	It is shown in \citet[Theorem 4.39]{follmer08Stochastic} that  the expected shortfall has a dual representation given by
	\begin{equation*}
		\av(f;\mu)= \inf_{\gamma\geq 0}\{\gamma + \alpha^{-1}\E_{\mu}[(f(X)-\gamma)^{+}]\}.
	\end{equation*}
	In particular, this implies \(\mu\mapsto  \av(f;\mu)\) is a concave functional.
	By Remark~\ref{rmk-4-ext}, we derive
	\begin{align*}
		\bav(f;\mu) & =\inf_{\lambda\geq 0} \{L^{*}{\lambda}+\sup_{\nu\in\scrP(\cX)}\{\av(f;\mu)-\cT_{\c}(\mu,\nu)\}\}                                                         \\
		            & =\inf_{\lambda\geq 0}\{L^{*}(\lambda)+\sup_{\pi\in\Pi_{\c}(\mu,*)}\inf_{\gamma\geq0}\{\gamma +\E_{\pi}[\alpha^{-1}(f(Y)-\gamma)^{+}-\lambda c(X,Y)]\}\}.
	\end{align*}
	Now it suffices to  show that we can swap the inner \(\sup\) and \(\inf\) and apply the same arguments as in Theorem~\ref{thm-4-disc}.
	We fix \(\lambda\geq 0\) and consider the first case where
	\begin{equation} \label{eqn-fs}\sup_{\pi\in\Pi_{\c}(\mu,*)}\E_{\pi}[\alpha^{-1}f(Y)-\lambda c(X,Y)]=+\infty.\end{equation}
	The inequality
	\begin{align*}
		\gamma + \E_{\pi}[\alpha^{-1}(f(Y)-\gamma)^{+}-\lambda c(X,Y)]\geq \gamma -\alpha^{-1}\gamma +\E_{\pi}[\alpha^{-1}f(Y)  -\lambda c(X,Y)]
	\end{align*}
	yields
	\begin{equation*}
		\inf_{\gamma\geq 0}\sup_{\pi\in\Pi_{\c}(\mu,*)}\{\gamma + \E_{\pi}[\alpha^{-1}(f(Y)-\gamma)^{+}-\lambda c(X,Y)]\}=+\infty.
	\end{equation*}
	On the other hand, we take a sequence of \(\pi_{n}\in\Pi_{\c}(\mu,*)\) such that \(\E_{\pi_{n}}[(\alpha^{-1}f(Y)-\lambda c(X,Y))]\) goes to infinity.
	We write \(\pi_{n}(\md x,\md y)=\mu(\md x) \eta_{n}(x,\md y )\) and define
	\begin{equation*}
		\tilde{\pi}_{n}(\md x,\md y)=\mu(\md x)[\alpha\eta_{n}(x,\md y)+(1-\alpha)\delta_{x}(\md y)].
	\end{equation*}
	It is direct to verify that \(\tilde{\pi}_{n}\) is again causal as it is a linear combination of two causal couplings.
	We further notice
	\begin{align*}
		 & \quad \inf_{\gamma\geq0}\{\gamma+\E_{\tilde{\pi}_{n}}[\alpha^{-1}(f(Y)-\gamma)^{+}-\lambda c(X,Y)]\}                                     \\
		 & =\inf_{\gamma\geq 0} \{\gamma+\alpha^{-1}\E_{\tilde{\pi}_{n}}[(f(Y)-\gamma)^{+}]\} - \lambda \alpha\E_{\pi_{n}}[c(X,Y)]                  \\
		 & \geq\sup_{0\leq \xi\leq \alpha^{-1}, \E_{\tilde{\pi}_{n}}[\xi]\leq 1} \E_{\tilde{\pi}_{n}}[f(Y)\xi] - \lambda \alpha\E_{\pi_{n}}[c(X,Y)] \\
		 & \geq \alpha\E_{\pi_{n}}[\alpha^{-1}f(Y) -\lambda c(X,Y)].
	\end{align*}
	The second equality follows from the duality representation of \(\av\);  the last inequality follows by taking \(\xi=\frac{\md \pi_{n}}{\md \tilde{\pi}_{n}}\in[0,\alpha^{-1}]\).
	Therefore, we deduce the duality holds with
	\begin{equation*}
		\sup_{\pi\in\Pi_{\c}(\mu,*)}\inf_{\gamma\geq 0}\{\gamma + \E_{\pi}[\alpha^{-1}(f(Y)-\gamma)^{+}-\lambda c(X,Y)]\}=+\infty.
	\end{equation*}

	The second case is to consider
	\begin{equation}
		\label{eqn-se}
		\sup_{\pi\in\Pi_{\c}(\mu,*)}\E_{\pi}[\alpha^{-1}f(Y)-\lambda c(X,Y)]<\infty.
	\end{equation}
	We write
	\begin{equation*}
		I:=\inf_{\gamma\geq 0} \sup_{\pi\in\Pi_{\c}(\mu,*)}\{\gamma + \E_{\pi}[\alpha^{-1}(f(Y)-\gamma)^{+}-\lambda c(X,Y)]\}.
	\end{equation*}
	Since \( \gamma_{\max}:=\sup_{\pi\in\Pi_{\c}(\mu,*)}\E_{\pi}[\alpha^{-1}f(Y)-\lambda c(X,Y)]<\infty\), it follows from \citet[Theorem 2]{fan53Minimax} that
	\begin{align*}
		I & =\inf_{\gamma\in[0,\gamma_{\max}]}\sup_{\pi\in\Pi_{\c}(\mu,*)}\{\gamma + \E_{\pi}[\alpha^{-1}(f(Y)-\gamma)^{+}-\lambda c(X,Y)]\}   \\
		  & =\sup_{\pi\in\Pi_{\c}(\mu,*)} \inf_{\gamma\in[0,\gamma_{\max}]}\{\gamma +\E_{\pi}[ \alpha^{-1}(f(Y)-\gamma)^{+}]-\lambda c(X,Y)\}.
	\end{align*}
	Let \(\pi_{\delta}\in\Pi_{\c}(\mu,*)\) such that
	\begin{equation*}
		\inf_{\gamma\in[0,\gamma_{\max}]}\{\gamma+ \E_{\pi_{\delta}}[\alpha^{-1}(f(Y)-\gamma)^{+}-\lambda c(X,Y)]\}\geq I-\delta.
	\end{equation*}
	Since the support of any probability measure on a Polish space is \(\sigma\)-compact, we may assume \(\cX_{n}\) is \(\sigma\)-compact for simplicity.
	In particular, we can write \(\cX_{n}=\lim_{m\to\infty}\cK^{m}_{n}\), where \(\{\cK^{m}_{n}\}_{m\geq 1}\) is ascending and compact.
	We write \(\cK^{m}=\cK^{m}_{0}\times \cK^{m}_{1}\times\cdots\times\cK^{m}_{N}\) and \(\Pi^{m}:=\{\pi\in\Pi_{\c}(\mu,*):\textrm{supp}(\pi)\subseteq \cX\times \cK^{m}\}\).
	Without loss of generality, we assume
	\begin{equation*}
		\label{eqn-4-k}
		\E_{\pi_{\delta}}[\1_{\cX\times(\cK^{m})^{c}}]+\E_{\pi_{\delta}}[\1_{\cX\times(\cK^{m})^{c}}(\alpha^{-1}f(Y)-\lambda c(X,Y))]\leq \frac{1}{m}.
	\end{equation*}
	We fix \(\tilde{x}\in\cK^{1}\) and define \(\tilde{\pi}=[x\mapsto (x,\tilde{x})]_{\#}\mu\).
	We construct a causal coupling \(\pi^{m}\) as
	\begin{equation*}
		\pi^{m}(\cdot)=\pi_{\delta}(\cdot\cap(\cX\times\cK^{m}))+  (1-\pi_{\delta}(\cX\times \cK^{m}))\tilde{\pi}(\cdot)\in\Pi^{m}
	\end{equation*}
	--- it follows from the fact that \(\pi_{\delta}(\cdot\cap(\cX\times\cK^{m}))/\pi_{\delta}(\cX\times \cK^{m})\) is causal.
	Therefore, we derive
	\begin{align*}
		 & \qquad \sup_{\pi\in\Pi^{m}}\inf_{\gamma\in [0,\gamma_{\max}]}\{\gamma + \E_{\pi}[\alpha^{-1}(f(Y)-\gamma)^{+}-\lambda c(X,Y)]\}                                      \\
		 & \geq \inf_{\gamma\in[0,\gamma_{\max}]}\{\gamma+ \E_{\pi^{m}}[\alpha^{-1}(f(Y)-\gamma)^{+}-\lambda c(X,Y)]\}                                                          \\
		 & \geq \inf_{\gamma\in[0,\gamma_{\max}]}\{\gamma+ \E_{\pi_{\delta}}[\alpha^{-1}(f(Y)-\gamma)^{+}-\lambda c(X,Y)]-\frac{1}{m}-\frac{1}{m}\E_{\mu}[\lambda c(X,x_{0})]\} \\
		 & \geq I -\delta -\frac{1}{m}(1+\E_{\mu}[\lambda c(X,x_{0})]).
	\end{align*}
	Here, the second last inequality follows from the estimate \eqref{eqn-4-k}.
	On the other hand,  \(\Pi^{m}\) is compact under the weak topology, see \citet[Theorem 3 (i)]{lassalle18Causal}.
	Therefore, we apply the minimax theorem from \citet{fan53Minimax} and derive
	\begin{align*}
		 & \quad \sup_{\pi\in \Pi_{\c}(\mu,*)}\inf_{\gamma\geq 0}\{\gamma + \E_{\pi}[\alpha^{-1}(f(Y)-\gamma)^{+}-\lambda c(X,Y)]\}                     \\
		 & \geq \lim_{m\to\infty} \sup_{\pi\in\Pi^{m}}\inf_{\gamma\geq 0}\{\gamma + \E_{\pi}[\alpha^{-1}(f(Y)-\gamma)^{+}-\lambda c(X,Y)]\}             \\
		 & = \lim_{m\to\infty} \inf_{\gamma\geq 0} \sup_{\pi\in\Pi^{m}}\{\gamma + \E_{\pi}[\alpha^{-1}(f(Y)-\gamma)^{+}-\lambda c(X,Y)]\}               \\
		 & =\lim_{m\to\infty}\inf_{\gamma \in [0,\gamma_{\max}]}\sup_{\pi\in\Pi^{m}}\{\gamma + \E_{\pi}[\alpha^{-1}(f(Y)-\gamma)^{+}-\lambda c(X,Y)]\}  \\
		 & =\lim_{m\to \infty} \sup_{\pi\in\Pi^{m}}\inf_{\gamma\in [0,\gamma_{\max}]}\{\gamma + \E_{\pi}[\alpha^{-1}(f(Y)-\gamma)^{+}-\lambda c(X,Y)]\} \\
		 & \geq I -\delta.
	\end{align*}
	The arbitrary choice of \(\delta\) allows us to conclude the proof.
\end{proof}

\begin{exam}
	\label{examp-es}
	We consider a two-step exotic option.
	Let \(X=(X_{0},X_{1},X_{2})\) be the underlying asset, \(f(x)=(x_{2}-x_{1}+1-K)^{+}\) the payoff of the option.
	We set interest rates and dividends to zero for simplicity.
	Let \(\mu\) be a reference pricing measure of the underlying.
	In this example we assume \(\mu\) is the finite marginal of a geometric Brownian motion \(S\) given by
	\begin{equation*}
		\md S_{t}=\sigma S_{t}\md W_{t},\quad S_{0}=1.
	\end{equation*}
	We set \(\alpha=0.95\), \(\sigma=0.2\),  \((X_{0},X_{1},X_{2})\sim (S_{0},S_{0.5},S_{1})\), \(c(x,y)=|x-y|^{2}\), and \(L=+\infty\1_{(0.3^{2},+\infty]}\).
	In Figure~\ref{fig-avar}, we plot the expected shortfall of the exotic option under different strikes.
	The classical expected shortfall \(\av\), the causal-DR expected shortfall \(\bav\), the W-DR expected shortfall are in solid, dashed, and dotted respectively.
	The gap between the solid and the dashed lines corresponds to the extra risk  coming from the model uncertainty.

	We emphasize that, in certain cases, restricting to non-anticipative model uncertainty does not lead to a reduction in risk.
	For example, consider the calendar spread payoff \(f(x)=(x_{2}-K)^{+}-(x_{1}-K)^{+}\) with any separable cost \(c\).
	\begin{figure}[htbp]
		\label{fig-avar}
		\centering
		\includegraphics[width=0.5\linewidth]{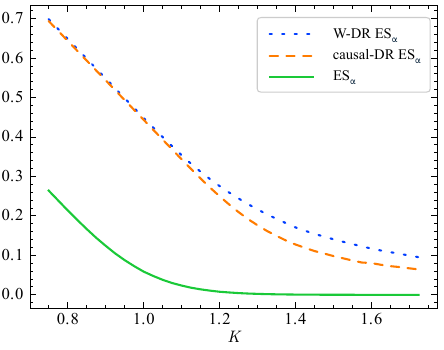}
		\caption{Comparison of  W-DR expected shortfall (dotted), causal-DR expected shortfall (dashed), and standard expected shortfall (solid) for an exotic option.}
	\end{figure}
\end{exam}

\begin{rmk}
	We can extend the risk-indifference pricing in \citet{xu06Risk} to a model misspecification context.
	For simplicity, we assume zero initial capital and zero liability.
	The  risk-indifference (sell) price is the minimum price a trader will charge so that the total risk of their portfolio does not increase.
	Here, we take the risk measure as the distributionally robust expected shortfall to counter the model uncertainty.
	The corresponding distributionally robust risk-indifference price is given by
	\begin{equation*}
		\inf_{H\in\bbH}\bav(f+(H {\circ} X)_{N};\mu)-\inf_{H\in\bbH}\bav((H {\circ} X)_{N};\mu),
	\end{equation*}
	where \(\bbH\) is the set of all predictable hedging strategies and \((H {\circ} X)_{N}=\sum_{n=1}^{N}H_{n}(X_{n}-X_{n-1})\) is the discrete stochastic integral.
	By Theorem~\ref{thm-avar}, we write
	\begin{align*}
		\inf_{H\in\bbH}\bav(f+(H {\circ} X)_{N};\mu)=\inf_{\lambda,\gamma\geq 0}\{L^{*}(\lambda)+\gamma+\inf_{H\in\bbH}U(\lambda,\gamma,H)\},
	\end{align*}
	where
	\begin{align*}
		 & \quad U(\lambda,\gamma)=                                                                                                                                                    \\
		 & \mu_{0}(\sup_{y_{0}\in\cX_{0}}\{ \cdots\mu_{x_{0:N-1}}(\sup_{y_{N}\in\cX_{N}}\{\alpha^{-1}(f(y)+ \sum_{n=1}^{N}H_{n}(y_{n}-y_{n-1})-\gamma)^{+}-\lambda c(x,y)\})\cdots\}).
	\end{align*}
	As \(\alpha\) goes to 0, it is known that the risk-indifference price under \(\av\) converges to the superhedging price.
	Under the current context, the robust superhedging price is given by
	\begin{equation*}
		\rho(f):=\inf\{x: x+(H {\circ} X)_{N}\geq f - L(\cT_{\c}(\mu,\nu))\quad \text{for some } H\in\bbH \quad \nu\text{-a.s.}\}.
	\end{equation*}
	We stress that this is a nondominated framework and is different from the  setting in \citet{bouchard15Arbitrage}.
	Indeed, we consider all possible   measures \(\nu\in\scrP(\cX)\) with a penalization \(L(\cT_{\c}(\mu,\nu))\), whereas in \citet{bouchard15Arbitrage} the authors considered a collection of measures which is stable under the concatenation of kernels.
	If \(\cX\) is compact and \(f\) is continuous, one can write down a pricing--hedging duality as
	\begin{equation*}
		\rho(f)=\sup_{\nu\in\scrP(\cX)}\sup_{\eta\ll \nu, \eta\in \scrM(\cX)}\{\E_{\eta}[f(X)]-L(\cT_{\c}(\mu,\nu))\},
	\end{equation*}
	where \(\scrM(\cX)\) denotes the set of martingale measures on \(\cX\).
	However, it is unclear if this holds under a more general setting, and we leave it for future work.
\end{rmk}

\section{Extension to optimal stopping problems}
\label{sec-4-os}
In this section, our aim is to derive a duality formula for causal-DR optimal stopping problems.
We work in a discrete-time setup, and let \(f_{n}:\cX_{0:n}\to \bbR\) for \(n\in I\) denote the payoff if the process stops at time \(n\).
We introduce the optimal stopping problem as
\begin{equation*}
	\os(f;\mu):= \sup_{\tau\in \scrT}\E_{\mu}[f_{\tau}(X_{0:\tau})],
\end{equation*}
where \(\scrT\) is the set of \(\bF\)-stopping times.
The corresponding causal-DR counterpart is given by
\begin{equation}
	\label{eqn-os}
	\bos(f;\mu):=\sup_{\nu\in\scrP(\cX)}\{\sup_{\tau\in\scrT}\E_{\nu}[f_{\tau}(X_{0:\tau})]-L(\cT_{\bc}(\mu,\nu))\}.
\end{equation}

\begin{figure}[htbp]
	\centering
	\includegraphics[width=\textwidth]{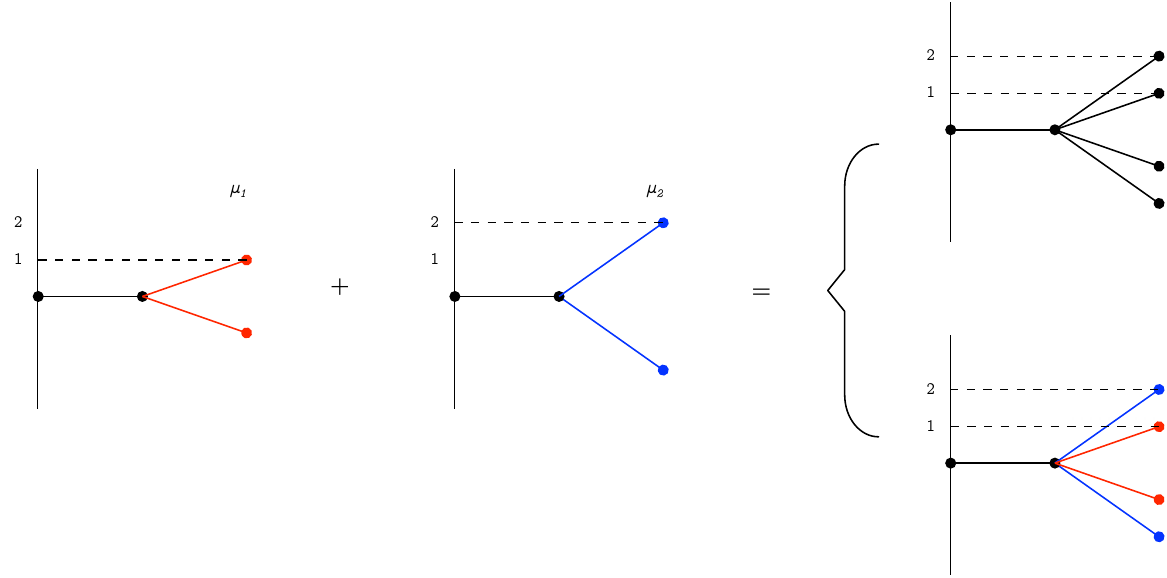}
	\caption{Different  averages of adapted stochastic processes \(\bbX_{1}=(\cX,\cF,\bF,\mu_{1},X)\) and \(\bbX_{2}=(\cX,\cF,\bF,\mu_{2},X)\). The upper process is a naturally filtered stochastic process, while the lower one is the process  carrying extra information from the coin toss.}
	\label{fig-4-os}
\end{figure}
Before we proceed, we first observe that \(\mu\mapsto \os(f;\mu)\) fails to be concave in general.
This prohibits us to apply Remark~\ref{rmk-4-ext} directly as in Section~\ref{sec-4-es}.
In Figure~\ref{fig-4-os}, we illustrate a concrete example by taking \(\mu_{1}=\frac{1}{2}(\delta_{(0,0,1)}+\delta_{(0,0,-1)})\), \(\mu_{2}=\frac{1}{2}(\delta_{(0,0,2)}+\delta_{(0,0,-2)})\), and \(f_{n}(x_{0:n})=|x_{n}^{2}-1|\).
Under this setting, it is direct to compute
\begin{equation*}
	\frac{1}{2}(\os(f;\mu_{1})+\os(f;\mu_{2}))=\frac{1}{2}(1+3)>\frac{3}{2}= \os(f;(\mu_{1}+\mu_{2})/2),
\end{equation*}
and hence \(\os(f;\mu)\) is not concave in \(\mu\).
In order to retrieve the concavity, we consider a relaxation of the control set \(\scrT\).
Heuristically, we can realize a process with law \(\frac{1}{2}(\mu_{1}+\mu_{2})\) by tossing a fair coin at the initial step and then follow the process with law \(\mu_{1}\) or \(\mu_{2}\) depending on the outcome of the coin.
If the stopping time were with respect to the natural filtration augmented by the coin toss, we would be able to stop the process depending on the outcome of the coin toss.
This would allow us to achieve an optimal stopping value exactly equal to \(\frac{1}{2}(\os(f;\mu_{1})+\os(f;\mu_{2}))\).

In light of the above arguments, we lift the causal-DR optimal stopping problem to the space of adapted processes \(\AP\) in order to allow richer filtrations.
We define
\begin{equation*}
	\bbos(f;\bbX):=\sup_{\bbY\in\AP}\{\sup_{\tau\in \scrT^{\bbY}} \E_{Q}[f_{\tau}(Y_{0:\tau})]-L(\cT_{\bc}(\bbX,\bbY))\},
\end{equation*}
where \(\bbX=(\cX,\cF,\bF,\mu,X)\), \(\bbY=(\Omega^{\bbY},\cF^{\bbY},\bF^{\bbY},\Q,Y)\), and \(\scrT^{\bbY}\) the set of \(\bF^{\bbY}\)-stopping times.
It is direct to verify \(\bos(f;\mu)\) is equivalent to the above formulation if we replace the optimizing set \(\AP\) by \(\NP\).
One may argue that on \(\AP\) there is no natural linear structure as different adapted stochastic processes may live on different filtered probability spaces.
To resolve this issue, we recall the nested space and the nested distribution introduced in \citet{bartl24Wasserstein}.

\begin{defn}[Nested space]
	We recursively define \(\widehat{X}_{N}=\cX_{N}\) and
	\begin{equation*}
		\widehat{\cX}_{n}=\widehat{\cX}_{n}^{-}\times\widehat{\cX}_{n}^{+}:=\cX_{n}\times\scrP(\widehat{\cX}_{n+1}) \quad\text{ for }0\leq n\leq N-1.
	\end{equation*}
	For any \(\hat{x}_{n}\in \widehat{\cX}_{n}\), we write it as \(\hat{x}_{n}=(\hat{x}_{n}^{-},\hat{x}_{n}^{+})\) with \(\hat{x}_{n}^{-}\in \cX_{n}\) and \(\hat{x}_{n}^{+}\in \scrP(\widehat{\cX}_{n+1})\).
	We say  \(\widehat{\cX}=\widehat{\cX}_{0}\) is the nested space associated to \(\cX\).
\end{defn}

We naturally extend the nested distribution to \(\AP\).
\begin{defn}[Nested distribution]
	\label{defn-4-nested}
	For a  given adapted stochastic process \(\bbX=(\Omega^{\bbX},\cF^{\bbX},\bF^{\bbX},X,P^{\bbX})\in\AP\), the associated information process \(\widehat{X}\) is given recursively by \(\widehat{X}_{N}:=X_{N}\) and
	\begin{equation*}
		\widehat{X}_{n}=(\widehat{X}_{n}^{-},\widehat{X}_{n}^{+}):=(X_{n},\Law(\widehat{X}_{n+1}^{\bbX}|\cF_{n}^{\bbX})).
	\end{equation*}
	We say  \(\widehat{P}^{\bbX}= \Law(\widehat{X}_{0})\in \scrP(\widehat{\cX})\) is the nested distribution associated to \(\bbX\).
\end{defn}

\begin{asmp}
	\label{asmp-4-os}
	We assume there exists \(p\geq 1\) such that the following conditions hold:
	\begin{enumerate}[label=(\roman*)]
		\item \(L:\bbR^{*}\to\bbR^{*}\) is non-decreasing and continuous on its domain \(\{L<\infty\}\).
		\item \(c:\cX\times\cX\to\bbR\) is continuous and has a polynomial growth
		      \[|c(x,y)|\leq C(1+d_{\cX}(\tilde{x},x)^{p}+d_{\cX}(\tilde{x},y)^{p})\quad \text{ for some } \tilde{x}\in\cX.\]
		\item \(f_{n}:\cX_{0:n}\to \bbR\) is continuous and has a polynomial growth \[|f_{n}(x_{0:n})|\leq C(1+d_{\cX}(\tilde{x}_{0:n},x_{0:n})^{p})\quad \text{ for some } \tilde{x}_{0:n}\in\cX_{0:n}.\]
	\end{enumerate}
\end{asmp}

The following result shows that  the bi-causal optimal transport problem between two adapted stochastic processes can be solved by a dynamic programming principle, and that this is equivalent to a classical optimal transport problem on the nested space.
The proof is similar to \citet[Theorem 3.10]{bartl24Wasserstein}, and we postpone it to the end of this section.

\begin{prop}
	\label{prop-4-bt}
	Let \(\bbX=(\Omega^{\bbX},\cF^{\bbX},\bF^{\bbX},\P,X)\) and \(\bbY=(\Omega^{\bbY},\cF^{\bbY},\bF^{\bbY},\Q,Y)\) be two adapted stochastic processes.
	Under Assumption~\ref{asmp-4-os}, there exists a continuous function \(\hat{c}:\widehat{\cX}\times \widehat{\cX}\to \bbR\) such that
	\begin{equation*}
		\cT_{\bc}(\bbX,\bbY)=\inf_{\pi\in \Pi(\widehat{P},\widehat{Q})}\E_{\pi}[\hat{c}(\widehat{X},\widehat{Y})]:=\widehat{\cT}(\widehat{P},\widehat{Q}).
	\end{equation*}
\end{prop}

\begin{thm}
	\label{thm-4-os}
	We recursively define \(\hat{f}_{n}:\widehat{\cX}_{0:n}\to \bbR\) for \(n\in I\) as follows.
	Let \(\hat{f}_{N}(\hat{x}_{0:N}):=f_{N}(x_{0:N}^{-})\) and for \(n=N-1,\dots,0\)
	\begin{equation*}
		\hat{f}_{n}(\hat{x}_{0:n}):=\max \{f_{n}(\hat{x}_{0:n}^{-}),\hat{x}_{n}^{+}(\hat{f}_{n+1}(\hat{x}_{0:n},\cdot))\}.
	\end{equation*}
	We assume \(\cX\) has no isolated points.
	Under Assumption~\ref{asmp-4-os}, we have
	\begin{equation*}
		\bos(f;\mu)=\inf_{\lambda\geq 0}\{L^{*}(\lambda)+\hat{\mu}(\sup_{\hat{y}}\{\hat{f}(\hat{y})-\hat{c}(\hat{x},\hat{y})\})\},
	\end{equation*}
	where \(\hat{\mu}\in\scrP(\widehat{\cX})\) is the nested distribution associated to \(\bbX=(\cX,\cF,\bF,\mu,X)\) and \(\hat{c}\) is the cost function given in Proposition~\ref{prop-4-bt}.
\end{thm}
\begin{proof}
	\emph{Step 1.} We first show \(\bos(f;\mu)=\bbos(f;\bbX)\).
	Since we assume \(\cX\) has no isolated points, by \citet[Theorem 5.4]{bartl24Wasserstein} \(\NP\) is a dense subset of \(\AP\) in \(\AW_{p}\).
	By Assumption~\ref{asmp-4-os}  (ii) and \citet[Theorem 3.6]{eckstein24Computational},   \(\bbY\mapsto\cT_{\c}(\bbX,\bbY)\)  are continuous with respect to \(\AW_{p}\).
	Moreover, by Assumption~\ref{asmp-4-os} (iii) and \citet[Proposition 6.1 (ii)]{bartl24Wasserstein}, the value of the optimal stopping problem is continuous in \(\AW_{p}\).
	These properties yield
	\begin{align*}
		\bbos(f;\bbX) & =\sup_{\bbY\in\AP}\{\sup_{\tau\in \scrT^{\bbY}} \E_{Q}[f_{\tau}(Y_{0:\tau})]-L(\cT_{\bc}(\bbX,\bbY))\}              \\
		              & =\sup_{\bbY\in\NP}\{\sup_{\tau\in \scrT^{\bbY}} \E_{Q}[f_{\tau}(Y_{0:\tau})]-L(\cT_{\bc}(\bbX,\bbY))\}=\bos(f;\mu).
	\end{align*}

	\emph{Step 2.} We now derive the duality formula for \(\bbos(f;\bbX)\).
	Notice that, by the Snell envelope theorem, we have
	\begin{equation*}
		\sup_{\tau\in\scrT^{\bbY}}\E_{Q}[f_{\tau}(Y_{0:\tau})]=\E_{\widehat{Q}}[\hat{f}(\widehat{X})],
	\end{equation*}
	where \(\widehat{Q}\)  is the nested distribution associated to \(\bbY\).
	Combining this with Proposition~\ref{prop-4-bt} yields
	\begin{align*}
		\bbos(f;\bbX) & =\sup_{\bbY\in\AP}\{\sup_{\tau\in \scrT^{\bbY}} \E_{Q}[f_{\tau}(Y_{0:\tau})]-L(\cT_{\bc}(\bbX,\bbY))\}                              \\
		              & \leq \sup_{\widehat{Q}\in \scrP(\widehat{\cX})} \{\E_{\widehat{Q}}[\hat{f}(\widehat{X})]-L(\widehat{\cT}(\hat{\mu},\widehat{Q}))\}.
	\end{align*}
	On the other hand, for any \(\widehat{Q}\in\scrP(\widehat{\cX})\) we can construct an adapted stochastic process \(\bbY\) such that \(\widehat{Q}\) is the nested distribution associated to \(\bbY\).
	Indeed, we can take  \(\bbY=(\widehat{\cX}_{0:N},\cF,\bF, Q,  \widehat{X}^{-})\), where \(\cF=\cB(\widehat{\cX}_{0:N})\), \(\bF=(\cF_{n})_{n=0}^{N}\) with \(\cF_{n}=\sigma(\widehat{X}_{0:n})\),  \(\widehat{X}^{-}\) is the projection from \(\widehat{\cX}_{0:N}\) to \(\widehat{\cX}_{0:N}^{-}\), and
	\begin{equation*}
		Q(\md \hat{x}_{0}, \md \hat{x}_{1},\dots,\md \hat{x}_{N})= \widehat{Q}(\md \hat{x}_{0})\hat{x}_{0}^{+}(\md \hat{x}_{1})\dots \hat{x}_{N-1}^{+}(\md \hat{x}_{N}).
	\end{equation*}
	Therefore, we derive
	\begin{equation*}
		\bbos(f;\bbX)= \sup_{\widehat{Q}\in \scrP(\widehat{\cX})} \{\E_{\widehat{Q}}[\hat{f}(\widehat{X})]-L(\widehat{\cT}(\hat{\mu},\widehat{Q})\}.
	\end{equation*}
	The right-hand side is a static Wasserstein DRO problem.
	Applying Theorem~\ref{thm-4-disc} with \(N=0\), we derive the desired duality.
\end{proof}

\begin{proof}[Proof of Proposition~\ref{prop-4-bt}]
	We recursively define \(\hat{c}_{n}:\widehat{\cX}_{0:n}\times\widehat{\cX}_{0:n}\to\bbR\) for \(n\in I\).
	Let \(\hat{c}_{N}(\hat{x}_{0:N},\hat{y}_{0:N})=c(x^{-}_{0:N},y^{-}_{0:N})\) and
	\begin{equation*}
		\hat{c}_{n}(\hat{x}_{0:n},\hat{y}_{0:n})=\sup_{\pi\in \Pi(\hat{x}_{n}^{+},\hat{y}_{n}^{+})}\E_{\pi}[\hat{c}_{n+1}(\hat{x}_{0:n},\widehat{X}_{n+1},\hat{y}_{0:n},\widehat{Y}_{n+1})].
	\end{equation*}
	Notice that each \(\hat{c}_{n}\) is a classic optimal transport problem, and hence it is continuous by Assumption~\ref{asmp-4-os} (ii) and \citet[Corollary 3.6]{bogachev21Optimal}.
	We recall the definition of \(\widehat{X}_{n}\) in Definition~\ref{defn-4-nested}
	\begin{equation*}
		\widehat{X}_{n}=(\widehat{X}_{n}^{-},\widehat{X}_{n}^{+}):=(X_{n},\Law(\widehat{X}_{n+1}^\bbX|\cF_{n}^{\bbX})).
	\end{equation*}
	Now we notice
	\begin{align*}
		\inf_{\pi\in \Pi_{\bc}(\bbX,\bbY)}\E_{\pi}[c(X,Y)] & = \inf_{\pi\in \Pi_{\bc}(\bbX,\bbY)}\E_{\pi}[\hat{c}_{N}(\widehat{X},\widehat{Y})]                                                                                                                  \\
		                                                   & = \inf_{\pi\in \Pi_{\bc}(\bbX,\bbY)}\E_{\pi}[\E_{\pi}[\hat{c}_{N}(\widehat{X},\widehat{Y})|
			                                                                                                  \cF_{n-1}^{\bbX}\otimes \cF_{n-1}^{\bbY}]                                                                                                          \\
		                                                   & =\inf_{\pi\in \Pi_{\bc}(\bbX,\bbY)}\E_{\pi}[\sup_{\gamma\in \Pi(\widehat{X}_{n}^{+},\widehat{Y}_{n}^{+})}\E_{(\xi,\eta)\sim \gamma}[\hat{c}_{N}(\widehat{X}_{0:N-1},\xi,\widehat{Y}_{0:N-1},\eta)]] \\
		                                                   & =\inf_{\pi\in \Pi_{\bc}(\bbX,\bbY)}\E_{\pi}[\hat{c}_{N-1}(\widehat{X}_{0:N-1},\widehat{Y}_{0:N-1})].
	\end{align*}
	The last equality follows from an extension of \citet[Proposition 2.4]{backhoff-veraguas17Causal}.
	We apply the above argument repeatedly, and obtain
	\begin{equation*}
		\cT_{\bc}(\bbX,\bbY)=\inf_{\pi\in \Pi(\hat{P},\hat{Q})}\E_{\pi}[\hat{c}(\widehat{X},\widehat{Y})],
	\end{equation*}
	where we set \(\hat{c}=\hat{c}_{0}\).
\end{proof}

\section*{Acknowledgement}
The author would like to thank Prof. Jan Ob\l\'oj  and Dr.~Fang~Rui~Lim for their support and insightful discussions.
The author acknowledges support by the EPSRC Centre for Doctoral Training in Mathematics of Random Systems: Analysis, Modelling, and Simulation (EP/S023925/1).

\bibliography{mybib}

\end{document}